\documentclass[10pt]{amsart}
\usepackage{latexsym, amssymb, stmaryrd, eucal,epic}

\newtheorem{thm}{Theorem}[section]
\newtheorem{lemma}[thm]{Lemma}

\newtheorem{cor}[thm]{Corollary}

\newtheorem{fact}[thm]{Fact}
\newtheorem{rmk}[thm]{Remark}

\newtheorem{question}[thm]{Question}
\newtheorem{convention}[thm]{Convention}

\theoremstyle{definition}
\newtheorem{df}[thm]{Definition}

\newcommand{\Los}{{\L}o{\'s}}

\newcommand{\BC}{\mathbb{C}}

\newcommand{\BE}{\mathbb{E}}
\newcommand{\BF}{\mathbb{F}}

\newcommand{\BK}{\mathbb{K}}
\newcommand{\BM}{\mathbb{M}}
\newcommand{\BN}{\mathbb{N}}

\newcommand{\BQ}{\mathbb{Q}}

\newcommand{\FM}{\mathfrak{M}}

\newcommand{\cu}[1]{\mathcal{#1}}

\newcommand{\bo}[1]{\boldsymbol{#1}}

\newcommand{\tp}{\operatorname{tp}}

\newcommand{\Th}{\operatorname{Th}}

\newcommand{\Sat}{\operatorname{Sat}}
\newcommand{\lcf}{\operatorname{lcf}}
\newcommand{\lca}{\operatorname{lca}}
\newcommand{\miin}{\operatorname{min}}
\newcommand{\snm}{\operatorname{stb\neg min}}

\newcommand{\Av}{\operatorname{Av}}

\newcommand{\stb}{\operatorname{stb}}

\def \lek{\trianglelefteq}
\def \eqk{\trianglelefteq\trianglerighteq}
\def \lk{\vartriangleleft}

\def \l{\llbracket}
\def \rr{\rrbracket}

\renewcommand{\hat}{\widehat}

\makeatletter

\def\indsym#1#2{%
  \setbox0=\hbox{$\m@th#1x$}%
  \kern\wd0%
  \hbox to 0pt{\hss$\m@th#1\mid$\hbox to 0pt{$\m@th#1^{#2}$}\hss}%
  \lower.9\ht0\hbox to 0pt{\hss$\m@th#1\smile$\hss}%
  \kern\wd0}
\newcommand{\ind}[1][]{\mathop{\mathpalette\indsym{#1}}}

\def\nindsym#1#2{%
  \setbox0=\hbox{$\m@th#1x$}%
  \kern\wd0%
  \hbox to 0pt{\hss$\m@th#1\not$\kern1.4\wd0\hss}
  \hbox to 0pt{\hss$\m@th#1\mid$\hbox to 0pt{$\m@th#1^{\,#2}$}\hss}%
  \lower.9\ht0\hbox to 0pt{\hss$\m@th#1\smile$\hss}%
  \kern\wd0}

\def\dotminussym#1#2{%
  \setbox0=\hbox{$\m@th#1-$}%
  \kern.5\wd0%
  \hbox to 0pt{\hss\hbox{$\m@th#1-$}\hss}%
  \raise.6\ht0\hbox to 0pt{\hss$\m@th#1.$\hss}%
  \kern.5\wd0}
\newcommand{\dotminus}{\mathbin{\mathpalette\dotminussym{}}}

\def\dotlesym#1#2{%
  \setbox0=\hbox{$\m@th#1-$}%
  \kern.5\wd0%
  \hbox to 0pt{\hss\hbox{$\m@th#1\le$}\hss}%
  \raise 1\ht0\hbox to 0pt{\hss$\m@th#1.$\hss}%
  \kern.5\wd0}
\newcommand{\dotle}{\mathbin{\mathpalette\dotlesym{}}}

\begin{document}

\title{Using Ultrapowers to Compare Continuous Structures}

\author{H. Jerome Keisler}

\address{University of Wisconsin-Madison, Department of Mathematics, Madison,  WI 53706-1388}
\email{keisler@math.wisc.edu}

\date{\today}

\begin{abstract}
In 1967 the author introduced a pre-ordering of all first order complete theories where T is lower than U if it is
easier for an ultrapower of a model of $T$ than an ultrapower of a model of $U$ to be saturated.
In a long series of recent papers, Malliaris and Shelah showed that this pre-ordering is very rich and gives a useful way of classifying simple theories.
In this paper we investigate the analogous pre-ordering in continuous model theory.
\end{abstract}

\maketitle

\setcounter{tocdepth}{2}

\tableofcontents

\section{Introduction}

The results in this paper bring together two research projects that began in the 1960's,  the so-called ``Keisler order'' $\lek$, and continuous model theory.
Both of these projects were mostly dormant for four decades, and then suddenly
blossomed into active and highly successful programs during the last decade.

In the paper [Ke67], the relation $\lek$ was introduced and proposed as a way of classifying  first order theories\footnote{In this paper, ``theory'' always means ``complete theory''.}
 with possibly different countable vocabularies.
Informally, $T\lek U$ if  is it is easier for an ultrapower of a model of $T$ than an ultrapower of a model of $U$ to be saturated.  Formally,  $T\lek U$
means that for every cardinal $\lambda$, every  regular ultrafilter $\cu D$ over a set of cardinality $\lambda$, and every model $\cu M$ of $T$ and $\cu N$ of $U$, if the ultrapower $\cu N_{\cu D}$ is $\lambda^+$-saturated
then so is the ultrapower $\cu M_{\cu D}$.  $T$ and $U$ are $\lek$-equivalent if $T\lek U$ and $U\lek T$.  The relation $\lek$ is obviously transitive.
It is shown in [Ke67] that $T\lek T$, so $\lek$ is a pre-ordering  on the set of all first order theories.
In a remarkable series of recent papers,
Malliaris and Shelah clarified the behavior of $\lek$ on first order theories.  They showed that $\lek$ gives a
useful classification of simple theories, that is so rich that it reveals $2^{\aleph_0}$ kinds of simple theories with countable vocabularies.

The modern treatment of continuous model theory deals with metric structures.
A surprisingly large part of first order model theory can be generalized to the continuous case.
Many notions and results in first order model theory have analogues for metric structures, including ultrapowers and $\kappa$-saturated structures.
In recent years, the model theory of metric structures has had exciting mathematical applications in analysis.

In this paper we study the analogue of the relation $\lek$ on the class of metric theories (complete theories of metric structures).
Some care is needed in our terminology, because we will be comparing metric theories with possibly different vocabularies.

Metric structures are defined as follows.
By  a vocabulary $V$ we mean a countable set of constant, function,  predicate symbols.   A continuous formula in  $V$
is built from atomic formulas using continuous functions on the real unit interval $[0,1]$ as connectives, and the quantifiers $\inf$ and $\sup$.
A \emph{real-valued structure with vocabulary} $V$ (briefly, a \emph{$V$-structure}) is like a first order structure except that the formulas take truth values in $[0,1]$, with $0$ representing truth.
A \emph{metric signature} $L$ over  $V$ specifies a distinguished binary predicate symbol $d\in V$, called the distance predicate of $L$, and a uniform continuity bound for each function and predicate symbol in $V$.
An \emph{$L$-metric structure} is a $V$-structure  $\cu M$ in which the distance predicate of $L$ is a complete metric and all the function and predicate symbols of $V$ respect the
bounds of uniform continuity.
First order structures with equality are just $L$-metric structures where the distance predicate of $L$ is the discrete metric and all atomic formulas have truth values in $\{0,1\}$.
We  define a \emph{metric structure} to be a real-valued structure $\cu M$ that is an $L$-metric structure for some metric signature $L$.

We next define what we mean by a metric theory.
By a $V$-\emph{sentence} we mean a continuous sentence in the vocabulary $V$.
Given an $L$-metric structure $\cu M$ where $L$ is a metric signature over a vocabulary $V$, the \emph{theory of} $\cu M$ is the set $\Th(\cu M)$ of all $V$-sentences
that have truth value $0$  in $\cu M$.
We say that $T$ is a \emph{metric theory} if $T=\Th(\cu M)$ for some metric structure $\cu M$.  If $\cu M$ is a metric structure and $T=\Th(\cu M)$, we call $\cu M$ a \emph{metric model} of $T$.

We say that a regular ultrafilter $\cu D$ over a set $I$ of cardinality $\lambda$ \emph{saturates} a metric structure $\cu M$ if the ultrapower $\cu M_\cu D$ is $\lambda^+$-saturated.
For two metric theories $T, U$ with possibly different  vocabularies, we write $T\lek U$ if  for every metric model $\cu M$ of $T$ and $\cu N$ of $U$, every regular ultrafilter
that saturates $\cu N$ saturates $\cu M$.  We will prove the following results.

 \begin{itemize}
\item[A.] Theorem \ref{t-equiv-triangle}: For any metric theory $T$ (with a countable vocabulary), we have $T\lek T$.  This means that for any two metric models $\cu M, \cu N$ of $T$,
$\cu M$ and $\cu N$ are saturated by the same regular ultrafilters.
It follows that $\lek$ is a pre-ordering (reflexive and transitive), and that
 $T\lek U$ if and only if there exist metric models $\cu M$ of $T$ and $\cu N$ of $U$ such that every ultrafilter that saturates $\cu N$ saturates $\cu M$.
\end{itemize}

 Note that  from the set of sentences $\Th(\cu M)$ one can recover the vocabulary $V$ but not the metric signature $L$.
In fact, it is possible for a metric theory $T$ in the sense of this paper to have many metric models with different  distance predicates in $V$.
By Result A, one can and should view $\lek$ as a relation between metric theories.

\begin{itemize}
\item[B.] Theorem \ref{t-stable-vs-unstable}: For metric theories $T$ and $U$, if $T$ is stable and $U$ is unstable then $T\lk U$.
\item[C.] Theorem \ref{t-stable-notmin}:  There are just two $\lek$-equivalence classes of stable metric theories, the $\lek$-minimal theories and the others.
\item[D.]  Theorem \ref{t-SOP2}: If $T$ satisfies the metric analogue of the property SOP$_2$ then $T$ is $\lek$-maximal among all metric theories.
\item[E.] Theorem \ref{t-rg-minimal}: The first order theory $T_{rg}$ of the random graph is $\lek$-minimal among all unstable metric theories.
\item[F.]  Let $T^R_{rg}$ be the randomization of $T_{rg}$.  (A model of $T^R_{rg}$ is built from a model of $T_{rg}$
by replacing elements by random elements).  Theorem  \ref{t-min-rgR}:
$T^R_{rg}$ is $\lek$-minimal among all metric theories with  TP$_2$, the tree property of the second kind.  Moreover,
$T_{rg}\lk T^R_{rg}$, and $T^R_{rg}$ is $\lek$-equivalent to the first order theory called $T^*_{feq}$.
\end{itemize}

These results are summarized in Figure 1.   The increasing direction in the pre-ordering is  toward the upper right.
The lower left quadrant of Figure 1 has the stable theories, the right half has the $\lek$-maximal theories, and the top half has
the theories with the independence property (IP).  Every unstable theory has either IP or SOP$_2$.
It is open whether or not $T^R_{rg}$ is $\lek$-maximal, but Figure 1 is drawn as if it is not.
The results A--C were announced in the author's plenary lecture at the March, 2020 meeting of the Association for Symbolic Logic, which was cancelled by the covid-19 pandemic.

\begin{figure}
\begin{center}
\setlength{\unitlength}{1mm}
\begin{picture}(100,100)(0,10)
\put(10,10){\line(1,0){80}}
\put(10,10){\line(0,1){80}}
\put(10,90){\line(1,0){80}}
\put(90,10){\line(0,1){80}}
\put(10,50){\line(1,0){40}}
\dottedline[.]{2}(52,50)(88,50)
\put(50,10){\line(0,1){80}}
\dottedline[.]{2}(52,88)(88,52)
\put(10,90){\line(1,-1){40}}
\put(10,50){\line(1,-1){40}}
\put(60,65){\makebox(0,0){SOP$_2$}}
\put(75,75){\makebox(0,0){?}}
\put(20,26){\makebox(0,0){min}}
\put(20,20){\makebox(0,0){stable}}

\put(35,35){\makebox(0,0){stable}}
\put(35,40){\makebox(0,0){$\neg$min}}
\put(35,84){\makebox(0,0){TP$_2$}}
\put(35,80){\makebox(0,0){unstable}}
\put(25,66){\makebox(0,0){$\neg$TP$_2$}}
\put(25,62){\makebox(0,0){unstable}}
\put(30,56){\makebox(0,0){$2^{\aleph_0}$ classes}}
\put(14,54){\makebox(0,0){$T_{rg}$}}
\put(10,50){\oval(17,17)[tr]}
\put(30,70){\oval(12,12)[tr]}
\put(24,76){\line(1,0){6}}
\put(36,64){\line(0,1){6}}

\put(32,72){\makebox(0,0){$T^R_{rg}$}}
\put(70,30){\makebox(0,0){SOP$_2$}}
\put(30,95){\makebox(0,0){$\neg$max}}
\put(70,95){\makebox(0,0){max}}
\put(100,70){\makebox(0,0){IP}}
\put(100,30){\makebox(0,0){$\neg$IP}}
\end{picture}
\end{center}
\caption{The $\lek$ order on metric theories}
\end{figure}
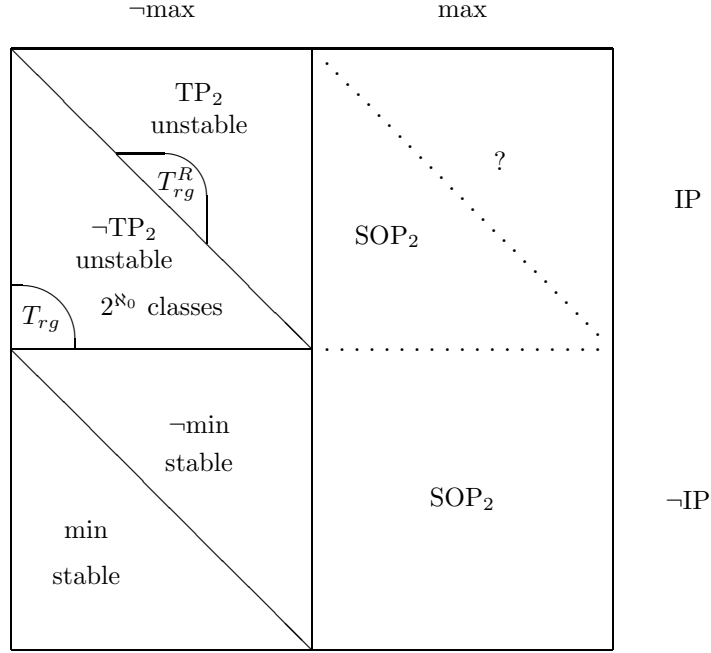

Historically, continuous model theory for real-valued structures was introduced by C. C. Chang and the author in 1966 in the monograph [CK66].  Four decades later, the 2008 paper [BBHU08] introduced the modern treatment of continuous model theory for metric structures.
In [CK66] the real-valued structures had an equality predicate.  In the current treatment, the real-valued structures do not have equality, and are the continuous analogue of
first order structures without equality.  Metric structures are the continuous analogue  of first order structures with equality, with the  distance predicate playing the
role of equality.

The ultraproduct construction for first order structures was introduced by \Los \ [Lo55].  Frayne, Morel and Scott [FMS62] used ultraproducts modulo regular ultrafilters.
The first order analogue of the result A appeared in Keisler [Ke67], and the first order analogues of B and C appeared in Shelah [Sh78].
For the next three decades, there was no further progress on the relation $\lek$ on first order theories until a breakthrough by Malliaris in [Ma09].
In that paper, Malliaris  introduced notions, such as the distribution of a set of formulas in an ultraproduct, that were key to
subsequent work on the $\lek$-order.  The main result in [Ma09]
showed that an ultrapower of a first order structure modulo a regular $\cu D$ is $\lambda^+$-saturated if, for every first order formula $\varphi$, the ultrapower is $\lambda^+$-saturated for $\varphi$-types.

The first order analogues of the results D, E, and F appeared in the following papers.
D: Malliaris and Shelah [MS16a]; E: Malliaris [Ma12]; and F: Malliaris and Shelah [MS13].
Those results show that the relation $\lek$ on first order theories also looks like Figure 1 (with $T^*_{feq}$ instead of $T^R_{rg}$).  The proofs of the first order results required the construction of ultrafilters with properties like being good.
Kunen's method of constructing ultrafilters via independent families of sets played a key role.
Good ultrafilters were introduced in [Ke64] and constructed assuming the generalized continuum hypothesis (GCH). In 1972, Kunen [Ku72]
used independent families of sets to construct good ultrafilters without the GCH.

The work of Malliaris and Shelah involved two parallel lines of research: the construction of ultrafilters with special set-theoretic properties,
and the interplay between the $\lek$ relation and model-theoretic properties of first order theories.

The results in this paper are obtained by extending the second of the above lines of research to metric theories.
Along the way we obtain other results
about metric structures that may be of interest in their own right.  Stability  and indiscernible sequences in metric theories
have been studied extensively in many papers (see Section \ref{s-stable} below).
However, as by-products of his proof of the results B and C above for first order theories, Shelah in [Sh78] proved many additional results about indiscernible sequences and the finite cover property,
which are generalized to the continuous case here.
Metric theories with the property TP$_2$  and the independence property have been studied in the paper [BY13].  The notion of an SOP$_2$ metric theory here may be new.

In the continuous model theory literature, a metric signature is usually held fixed and the ultraproduct is introduced for families of metric structures with the
same metric signature.
To prove the results in this paper we will need a more general  ultraproduct construction, for families of real-valued structures with the same vocabulary.
That leads us to consider the $\lek$ relation between theories of arbitrary real-valued structures, rather than just between metric theories.

It turns out that the pre-ordering $\lek$ is essentially the same in that case.
For a real-valued structure $\cu M$ with vocabulary $V$, $\Th(\cu M)$ is again defined as the set of all $V$-sentences with truth value $0$ in $\cu M$,
and we call $\Th(\cu M)$ a continuous theory with vocabulary $V$.  We have
$$ \{\mbox{first order theories}\}\subsetneqq\{\mbox{metric theories}\}\subsetneqq \{\mbox{continuous theories}\}.$$
We define the relation $T\lek U$ for continuous theories in the same way as for metric theories.
Theorem 3.6 below is a stronger version of result A above, showing that $T\lek T$ for every continuous theory $T$, so any two  real-valued models $\cu M, \cu N$  of $T$
(metric or not) are saturated by the same regular ultrafilters.  Theorem 5.1 below will show that every continuous theory is $\lek$-equivalent to a metric theory.  We may therefore
identify the $\lek$-equivalence classes of continuous theories with the $\lek$-equivalence classes of metric theories in a way that preserves the $\lek$ relation.

We thank Isaac Goldbring, James Hanson, and Maryanthe Malliaris for valuable discussions related to this work.

\section{Preliminaries on continuous model theory}

We refer to [BBHU08] and [BU09] for background material on metric and premetric structures.  See the Introduction above for the notion of a vocabulary $V$ and a metric signature $L$ over $V$.

\begin{convention}
Throughout this paper,  $V$ will denote a countable vocabulary.
To insure that the collection of all vocabularies is a set rather than a proper class, we also
require that the elements of a vocabulary are hereditarily finite sets.
\end{convention}

By a \emph{real-valued structure}  $\cu M$  with vocabulary $V$ and universe set $M$ (briefly, a $V$-structure)
we mean an object that has a non-empty universe set $M$, and interpretations:
\begin{itemize}
\item $c^{\cu M}\in M$ for each constant symbol $c\in V$,
\item $F^{\cu M}\colon M^n\to M$ for each $n$-ary function symbol $F\in V$,
\item $P^{\cu M}\colon M^n\to[0,1]$ for each $n$-ary predicate symbol $P\in V$.
\end{itemize}

We say that $\cu M$ is a real-valued structure if $\cu M$ is a $V$-structure for some $V$.

Let $L$ be a metric signature over $V$ with  distance predicate $d\in V$.  An \emph{$L$-premetric structure}  is a $V$-structure $\cu M$  such that
the interpretation $d^\cu M$  is a pseudo-metric on $M$, and for each predicate or function symbol $S\in V$,
the interpretation $S^{\cu M}$ respects the modulus of uniform continuity given by $L$.
An \emph{$L$-metric structure} is an $L$-premetric structure such that the interpretation $d^\cu M$ is a complete metric on $M$.
We say that $\cu M$ is a \emph{premetric structure} if there exists a metric signature $L$ such that $\cu M$ is an $L$-premetric structure.
 Similarly for metric structures.\footnote{In [BBHU08], premetric structures are called ``prestructures'', and metric structures are called ``structures''.
 In [BU09], vocabularies are called ``non-metric signatures''.  In  [Ke21], real-valued structures are called ``general structures''.}

The notion of an atomic formula is the same as in first order logic.
We assume familiarity with the notions of a continuous formula, a continuous sentence, and the truth value $\varphi^{\cu M}(\vec a)$ of a continuous formula $\varphi$
in a metric structure $\cu M$ at a tuple $\vec a$ in $M$.
By a $V$-formula we mean a continuous formula built from symbols in $V$.
The notion of truth value of a $V$-formula is exactly the same for arbitrary real-valued structures as for metric structures.
We write $\varphi(\vec v)$ to indicate that all the free variables of $\varphi$ are in $\vec v$, and use similar notation for sets of formulas.
A set  $\Phi(\vec v)$ of $V$-formulas  is said to be \emph{satisfiable in} $\cu M$ if there exist a $|\vec v|$-tuple $\vec a$ in M
such that $\varphi^{\cu M}(\vec a)=0$ for all $\varphi(\vec v)\in\Phi(\vec v)$.
A mapping $h\colon \cu M\to\cu N$ is an \emph{elementary embedding} if for every tuple $\vec a$ in $M$ and every  $V$-formula $\varphi(\vec v)$ we have
$\varphi^\cu M(\vec a)=\varphi^\cu N(h(\vec a)).$
The  elementary equivalence relation $\cu M\equiv\cu N$ and the elementary extension relation $\cu M\prec\cu N$
are defined as in [BBHU08], and are applied to arbitrary $V$-structures as well as to metric structures.

When a vocabulary $V$ is clear from the context, $V$-formulas will simply be called formulas.

Given a $V$-structure $\cu M$  and a (possibly uncountable) set $A\subseteq M$, we let $V_A=V\cup\{c_a\colon a\in A\}$ be the set formed by adding to $V$ a new constant
symbol $c_a$ for each element $a\in A$. $\cu M_A$ denotes the structure obtained from $\cu M$ by interpreting
 $c_a$ by $a$ for each $a\in A$.  The constant symbols $c_a, a\in A$, are called \emph{parameters from $A$}, or \emph{parameters from $\cu M$}.
By a formula  with parameters from $A$
we mean a formula obtained from a formula without parameters by replacing some of the free variables by parameters in $A$.
We say that a set $\Gamma$ of formulas with parameters from $A$ is \emph{satisfiable in} $\cu M$ if $\Gamma$ is  satisfiable in $\cu M_A$,
and \emph{finitely satisfiable in} $\cu M$ if every finite subset of $\Gamma$ is satisfiable in $\cu M_A$.

Recall from the Introduction that for a $V$-structure $\cu M$, the theory of $\cu M$ is  the set $\Th(\cu M)$ of all $V$-sentences $\varphi$
such that $\varphi^{\cu M}=0$.   Thus $Th(\cu M)=Th(\cu N)$ if and only if $\cu M\equiv\cu N$.
Also,  $T$ is a \emph{continuous theory} if $T=\Th(\cu M)$ for some real-valued structure $\cu M$,
and $T$ is a \emph{metric theory} if $T=\Th(\cu M)$ for some metric structure $\cu M$.

\begin{rmk}  \label{r-express}
 For each metric signature $L$ with distance predicate $d$, there is a set of $V$-sentences that expresses the fact that $d$ is a pseudo-metric and each function and predicate has
the required modulus of uniform continuity.
\end{rmk}

\begin{fact}  \label{f-premetric-metric}   \footnote{When we state a fact without giving a reference, the result can be found in [BBHU08].}
For every $L$-premetric structure $\cu M$, there is an $L$-metric structure $\cu N$,  the \emph{completion} of $\cu M$, such that $\cu N\equiv\cu M$, so $\Th(\cu M)$ is an $L$-metric theory.
\end{fact}

\begin{cor}  \label{c-metric-theory-char}
A $V$-theory $T$ is an $L$-metric theory if and only if $T$ contains the set of sentences mentioned in Remark \ref{r-express}.
\end{cor}

We say that $\cu M$ is a \emph{model of} a set of sentences $U$, in symbols $\cu M\models U$, if $\varphi^{\cu M}=0$ for all
$\varphi\in U$.  Thus $\cu M$ is a model of $Th(\cu M)$.  If $T=\Th(\cu M)$ and $\cu M\models U$, we also write $T\models U$.
We write $\cu M\models \varphi$, or $\cu M\models\varphi(\vec a)=0$, if $\varphi^\cu M(\vec a)=0$. We write $\cu M\models\varphi(\vec a)\le r$ if $\varphi^\cu M(\vec a)\le r$.

\begin{cor}  \label{c-model-premetric}  If $T$ is an $L$-metric theory, then every model of $T$  is an $L$-premetric structure.
\end{cor}

\begin{proof}  By Remark \ref{r-express}.
\end{proof}

\begin{fact}  \label{f-truth-value}
Suppose $\cu M$ is an $L$-premetric structure and $L$ has distance predicate $d$.   For any continuous formula $\varphi(\vec v)$ in the vocabulary of $\cu M$,
$\varphi^{\cu M}$ is a  mapping from $M^{|\vec v|}$ into $[0,1]$ that is uniformly continuous with respect to $(M,d^{\cu M})$.
\end{fact}

\begin{df}  \label{d-sat}
We say that a $V$-structure $\cu M$ is $\kappa$-saturated if for every set $A\subseteq M$ of cardinality $|A|<\kappa$, every set of  $V$-formulas
with parameters from $A$ that is finitely satisfiable in $\cu M$ is satisfiable in $\cu M$.
\end{df}

By a \emph{strict continuous formula} we mean a formula built from atomic formulas in finitely many
steps using only the quantifiers $\sup,\inf$, and the connectives $0,1,\min,\max, \cdot /2, \dotminus$ (where  $x\dotminus y=\max(x-y,0)$).

\begin{rmk} \label{r-dotminus}
\noindent\begin{itemize}
\item $|x-y|=\max(x\dotminus y,y\dotminus x).$
\item $x\dotminus y=0$ if and only if $x\le y$.
\item $\cu M\models\theta\le r$ if and only if $\cu M\models\theta\dotminus r$.
\item We sometimes use the alternate notation $\theta\dotle\psi$ for the formula $\theta\dotminus \psi$.
\end{itemize}
\end{rmk}

\begin{rmk}  \label{r-strict}
 The set of strict $V$-formulas is countable.
\end{rmk}

\begin{fact}  \label{f-strict-approx}  (See Theorem 6.3 in [BBHU08]).
For each $V$-formula $\varphi(\vec x)$ and $\varepsilon>0$ there is a strict $V$-formula $\theta(\vec x)$ such that
$|\varphi^\cu M(\vec x)-\theta^\cu M(\vec x)|<\varepsilon$ for every $\vec x$ in every $V$-structure $\cu M$.
\end{fact}

\begin{lemma}  \label{l-strict-sat}  If $\kappa>\aleph_0$, then a $V$-structure $\cu M$ is $\kappa$-saturated if and only if  every set of fewer than $\kappa$ formulas
with parameters from $M$ that is finitely satisfiable in $\cu M$ is satisfiable in $\cu M$.
\end{lemma}

\begin{proof}  By Remark \ref{r-strict}, there are countably many strict $V$-formulas without parameters.  Hence there are $|A|+\aleph_0$ strict $V$-formulas with parameters from $A$.
So if $|A|<\kappa$ then there are fewer than $\kappa$ strict $V$-formulas with parameters from $A$.  The result now follows from Fact \ref{f-strict-approx}.
\end{proof}

In [BBHU08], the ultraproduct of an indexed family of $L$-metric structures with the same metric signature $L$ modulo an ultrafilter is defined, and it is proved
to again be an $L$-metric structure.  In this paper, we will need the slightly more general notion of an ultraproduct of an indexed family of $V$-structures
with the same vocabulary $V$. We will use the following notion of  reduction  of $V$-structures from the paper [Ke21].

\begin{df}  \label{d-reduced}
For $a,b\in M$, we write $a\doteq^{\cu{M}} b$ if for every atomic formula $\varphi(x,\vec{z})$ and tuple $\vec{c}\in M^{|\vec z|}$, $\varphi^{\cu{M}}(a,\vec{c})=\varphi^{\cu{M}}(b,\vec{c}).$
A $V$-structure $\cu{M}$ is \emph{reduced} if whenever $a\doteq^\cu{M} b$ we have $a=b$.
\end{df}

The relation $\doteq^{\cu{M}}$ goes back to Leibniz around 1840,
The \emph{reduction map} for $\cu M$ is the mapping that sends each element of $M$ to its equivalence class under $\doteq^{\cu M}$.
The \emph{reduction} of the $V$-structure $\cu{M}$ is the reduced $V$-structure $\cu{N}$ such that $N$ is the set of equivalence classes of elements of
$M$ under $\doteq^{\cu{M}}$, and the reduction map for $\cu{M}$ is an elementary embedding of $\cu M$ onto $\cu N$.
We say that $\cu{M},\cu{M}'$ are \emph{isomorphic}, in symbols $\cu {M}\cong\cu{M}'$, if there is an elementary embedding from the reduction of $\cu{M}$ onto the reduction of $\cu{M}'$.
 The following fact
gives a useful characterization of reduced $V$-structures $\cu M$ in the case that $\cu M$ is a premetric structure.

\begin{fact}  \label{f-reduced-premetric} (See Theorem 3.7 in [BBHU08]). Let $\cu M$ be an  $L$-premetric structure and let $d$ be the  distance predicate of $L$.
$\cu M$ is reduced if and only if for all $x,y\in M$, if  $d^{\cu M}(x,y)=0$ then $x=y$.
\end{fact}

\begin{rmk}  \label{r-reduction}
\noindent\begin{itemize}
\item $\cong$ is an equivalence relation on $V$-structures.
\item Every $V$-structure is isomorphic to its reduction.
\item  If there is an elementary embedding of $\cu{M}$ \emph{onto} $\cu{N}$, then $\cu{M}\cong\cu{N}$.
\item $\cu{M}\cong\cu{N}$ implies $\cu{M}\equiv\cu{N}$.
\item A $V$-structure is $\kappa$-saturated if and only if its reduction is $\kappa$-saturated.
\end{itemize}
\end{rmk}

\begin{rmk}  \label{r-sat-metric}  Every reduced $\aleph_1$-saturated $L$-premetric structure $\cu M$ is an $L$-metric structure.
\end{rmk}

\begin{proof} By Fact \ref{f-reduced-premetric}, $(M,d^\cu M)$ is a metric space.  Completeness of the metric follows easily from $\aleph_1$-saturation.
\end{proof}

\medskip

\begin{convention} $\lambda$ will always denote an infinite cardinal, and $I$ will always be a set of cardinality $\lambda$.
When working with ultraproducts, we will use the notation $\langle r[t]\rangle_{t\in I}$ for a function $r$ with domain $I$.
\end{convention}

This convention will help keep things straight in the ultraproduct construction, by using $t$ as a variable over $I$ and leaving the symbol $i$ as a variable ranging over the natural numbers.

\begin{df}  \label{d-ultraroduct}
Let $\cu D$ be an ultrafilter over  $I$.  For each function $r\colon I\to[0,1]$, $\lim_{\cu D} r$ is the unique element $s\in[0,1]$
such that for each real $\varepsilon>0$, the set
$\{t\in I\colon |r[t]-s|<\varepsilon\}$ belongs to $\cu D$.
\end{df}

The following definition is taken from [Ke21].

\begin{df} \label{d-ultraproduct} Let $\cu{D}$ be an ultrafilter  over a set $I$ and $\cu{M}_t$ be a $V$-structure for each $t\in I$.
The \emph{pre-ultraproduct} $\prod^{\cu{D}}\cu{M}_t$ is the $V$-structure $\cu{M}'=\prod^{\cu{D}}\cu{M}_t$  such that:
\begin{itemize}
\item $M'=\prod_{t\in I} M_t$, the cartesian product.
\item For each constant symbol $c\in V$, $c^{\cu{M}'}=\langle c^{\cu{M}_t}\rangle_{t\in I}$.
\item For each $n$-ary function symbol $G\in V$ and $n$-tuple $\vec{a}=\langle \vec{a}[t]\rangle_{t\in I}$ in $M'$,
$$G^{\cu{M}'}(\vec{a})=\langle G^{\cu{M}_t}(\vec{a}[t])\rangle_{t\in I}.$$
\item For each $n$-ary predicate symbol $P\in V$ and $n$-tuple $\vec{a}$ in $M'$,
$$P^{\cu{M}'}(\vec{a})=\lim_{\cu{D}}\langle P^{\cu{M}_t}(\vec{a}[t])\rangle_{t\in I}.$$
\end{itemize}

The \emph{ultraproduct} $\prod_{\cu{D}}\cu{M}_t$ is the reduction of the pre-ultraproduct $\prod^{\cu{D}}\cu{M}_t$.
For each $a\in M'$ we also let $a_{\cu{D}}$ denote the equivalence class of $a$ under $\doteq^{\cu{M}'}$.
\end{df}

\begin{rmk} It follows from Fact \ref{f-reduced-premetric} that if $\cu M_i$ is  an $L$-metric structure with the same $L$ for each $i\in I$,
then $\prod_{\cu{D}}\cu{M}_t$ is exactly the ultraproduct of the metric structures $\cu M_i$ as defined in [BBHU08].
\end{rmk}

 The next fact is the analogue for $V$-structures of the fundamental theorem of \Los.  It is proved in the same way as the corresponding result for metric structures, Theorem 5.4 in [BBHU08].

\begin{fact} \label{f-Los-2} (\Los \ Theorem).
 Let $\cu{M}_t$ be a $V$-structure for each $t\in I$, let $\cu{D}$ be an ultrafilter over $I$, and let
$\cu{M}=\prod_{\cu{D}}\cu{M}_t$ be the ultraproduct.  Then for each formula $\varphi$ and tuple $\vec{b}$ in the cartesian product $\prod_{t\in I} M_t$,
$$\varphi^{\cu{M}}(\vec{b}_{\cu{D}})=\lim_{\cu{D}}\langle\varphi^{\cu{M}_t}(\vec{b}[t])\rangle_{t\in I}.$$
\end{fact}

\begin{cor}  \label{f-ultraproduct-isomorphic}  Suppose $\cu D$ is an ultrafilter over $I$, and for each $t\in I$, $\cu M_t, \cu N_t$ are $V$-structures.

(i) If  $\cu M_t\cong\cu N_t$
for each $t\in I$, then $\prod_{\cu D}\cu M_t\cong\prod_{\cu D}\cu N_t$.

(ii)  If  $\cu M_t\equiv\cu N_t$
for each $t\in I$, then $\prod_{\cu D}\cu M_t\equiv\prod_{\cu D}\cu N_t$.
\end{cor}

\begin{fact}  \label{f-ultrapower-saturated}   Every ultraproduct of $L$-metric structures with the same $L$ is an $L$-metric structure.
\end{fact}

One of the main reasons that metric signatures were introduced in the definition of a metric structure was to insure that Fact \ref{f-ultrapower-saturated} holds.

In the case that all the $\cu M_t$ are the same, $\cu M_t=\cu M$ for all $t\in I$, the pre-ultraproduct and ultraproduct are called
the \emph{pre-ultrapower} and  \emph{ultrapower} respectively, and are denoted by $\cu M^{\cu D}$ and $\cu M_{\cu D}$.

\begin{cor} \label{c-natural-embedding}   For every $V$-structure $\cu M$ and ultrafilter $\cu D$ over $I$, the diagonal embedding
$a\mapsto I\times \{a\}$ is an elementary embedding of $\cu M$ into $\cu M^{\cu D}$.
\end{cor}

\begin{df}
An ultrafilter $\cu D$ over $I$ is said to be \emph{regular} if  there is a subset $\cu X\subseteq \cu D$ of cardinality $\lambda$
such that each $t\in I$ belongs to only finitely many $X\in \cu X$.  We also say that $\cu X$ \emph{regularizes} $\cu D$.
\end{df}

\begin{fact} ([FMS62])
For each infinite set $I$ there exist regular ultrafilters over $I$.
\end{fact}

Hint: If $I$ is the set of all finite subsets of $J$, let $\cu X=\{X_j\colon j\in J\}$ where $X_j=\{Y\subseteq I\colon \{j\}\subseteq Y \} $.

\begin{rmk}  \label{r-regular-aleph-1}
Let $\cu D$ is a regular ultrafilter over $I$.
\begin{itemize}
\item[(i)]  If  $\cu M_t$ is a $V$-structure for each $t\in I$,
then the ultraproduct $\prod_{\cu D}\cu M_t$ is $\aleph_1$-saturated.
\item[(ii)]  If $\cu M_i$ is an $L$-premetric structure for each $i\in I$, then  the ultraproduct $\prod_{\cu D}\cu M_t$ is an $L$-metric structure.
\end{itemize}
\end{rmk}

\begin{proof} (i)  Theorem 6.1.1 of [CK12] gives the corresponding result in first order logic.
It is routine to modify the proof in [CK12] to obtain the result stated here.

(ii) By definition, $\prod_{\cu D}\cu M_t$ is a reduced structure.  By Remark \ref{r-express} and the \Los \ theorem, $\prod_{\cu D}\cu M_t$ is an $L$-premetric structure.  By Remark \ref{r-sat-metric},
$\prod_{\cu D}\cu M_t$ is $\aleph_1$-saturated.  Then by (i) above, $\prod_{\cu D}\cu M_t$ is an $L$-metric structure.
\end{proof}

By a \emph{two-valued structure}, or a \emph{first order structure without equality},  we mean a $V$-structure $\cu M$ such that for
every $n$-ary predicate symbol $P\in V$, $P^{\cu M}\colon M^n\to\{0,1\}$.
By a  \emph{first order structure} we mean a two-valued $L$-metric structure $\cu M$ such that:
\begin{itemize}
\item $L$ has a  distance predicate  $d$ and the trivial modulus of uniform continuity for each symbol in $V$.
\item $d^\cu M(x,y)$ is the discrete metric; $d^\cu M(x,y)=0$ if $x=y$, and $d^\cu M(x,y)=1$ otherwise.
\end{itemize}

By a \emph{first order formula} we mean a formula that is built from atomic formulas using only the quantifiers $\sup,\inf$, and the connectives $0,1,\min,\max, 1\dotminus x$.
Note that every first order formula is a strict continuous formula. A first order formula  is traditionally
written with the notation $ \forall, \exists, \top, \bot, \vee, \wedge, \neg$ instead of $\sup,\inf,0,1,\min,\max, 1\dotminus x$.
First order structures and formulas are sometimes called classical structures and formulas.

\begin{fact}  \label{f-cardinality-ultrapower}  ([FMS62].  See Proposition 4.3.7 of [CK12].)
If $\cu D$ is a regular ultrafilter over $I$ and $\cu M$ is an infinite first order structure, then $\cu M_\cu D$ has cardinality $|M|^{|I|}$.
\end{fact}

The following lemma is well-known and easily checked, but is stated explicitly here for completeness.

\begin{lemma}  \label{l-classical}
Let $\cu M$ be a two-valued structure.
\begin{itemize}
\item[(i)]  For every first order formula $\varphi(\vec v)$ and $|\vec v|$-tuple $\vec a$ in $\cu M$,
$\varphi^{\cu M}(\vec a)\in\{0,1\}$.
\item[(ii)] Suppose $(\cu M,\vec a)$ and $(\cu M,\vec b)$ satisfy the same first order sentences. Then $(\cu M,\vec a)$ and $(\cu M,\vec b)$ satisfy the same
continuous sentences, that is, $(\cu M,\vec a)\equiv(\cu M,\vec b)$ as $V$-structures.
\item[(iii)] Suppose $\cu M$ is a first order structure that is $\kappa$-saturated in the sense of first order model theory.
Then $\cu M$ is $\kappa$-saturated as a $V$-structure.
\end{itemize}
\end{lemma}

If $\cu M$ is a first order structure with vocabulary $V$, we abuse notation by letting $\Th(\cu M)$ denote the set of all first order $V$-sentences true in $\cu M$,
rather than the set of continuous $V$-sentences true in $\cu M$.
By a \emph{first order theory} we mean the set $\Th(\cu M)$ where $\cu M$ is a first order structure.

\section{Comparing two continuous theories}  \label{s-general}

\begin{convention} For the rest of this paper, $\cu D$ will always denote a regular ultrafilter over a set $I$ of cardinality $\lambda$.
\end{convention}

We now define the relation $\lek$ on the  set of all continuous theories, which contains the set of all metric theories.
 The definitions will be the same as for in first order model theory, but using
the continuous notions of ultrapower and saturation.

\begin{df}  \label{d-order} Let  $\cu M$ be a $V$-structure.  We say that $\cu D$ \emph{saturates} $\cu M$
if the ultrapower $\cu M_\cu D$ is $\lambda^+$-saturated (or equivalently, the pre-ultrapower $\cu M^\cu D$ is $\lambda^+$-saturated).

Given a continuous theory $T$, we say that $\cu D$ \emph{saturates} $T$, in symbols $\cu D\in\Sat(T)$, if $\cu D$ saturates every model of $T$.
\end{df}

\begin{df}

Given continuous theories $T, U$  with possibly different vocabularies, we write $T\lek U$ if $\Sat(U)\subseteq\Sat(T)$.

We write $ T\lk U$ if $T\lek U$ but not $U\lek T$.

We say that $T, U$ are $\lek$-\emph{equivalent}, in symbols
$T\eqk U$, if $T\lek U$ and $U\lek T$.
 \end{df}

\begin{rmk} $T\eqk U$ if and only if $\Sat(T)=\Sat(U)$.
\end{rmk}

By Lemma \ref{l-classical},  for first order theories, the relations $T\lek U$ and $T\eqk U$
are the same as in [Ke67].

Note that every ultraproduct of $V$-structures  is reduced.  By definition, the reduction map is an elementary embedding
from the pre-ultraproduct onto the ultraproduct.  We will usually work with the pre-ultraproduct rather than with the ultraproduct.
One reason for this choice is that pre-ultrapowers of $V$-structures commute with expansions.  If $\cu M$ is a $V$-structure and $W$ is a subset of $V$,  the
\emph{$W$-part}\footnote{The $W$-part is called the $W$-reduct in first order model theory (see [CK12]).  We say $W$-part here because the word reduction is used in
a different way in continuous model theory.}
 of $\cu M$ is the $W$-structure $N$ with the same universe as $\cu M$ such that for each symbol $S\in W$, $S^{\cu N}=S^{\cu M}$.  If $\cu N$ is
the $W$-part of $\cu M$,  we say that
$\cu M$ is an \emph{expansion} of $\cu N$.  In that case, the pre-ultrapower $\cu N^\cu D$ is exactly the $W$-part of the pre-ultrapower $\cu M^\cu D$, while
the ultrapower $\cu N_\cu D$ is only the reduction of the $W$-part of the ultrapower $\cu M_\cu D$.
One reason that ultraproducts of $V$-structures will be useful in studying the $\lek$ relation on metric theories is that it allows us to take ultraproducts of arbitrary expansions
of metric structures to a larger vocabulary.

In the next section we will prove the following theorem, which generalizes a result in [Ke67] from first order structures to  real-valued structures.

\begin{thm}  \label{t-equiv-triangle}
For any continuous theory $T$, and hence for every metric theory $T$, we have $T\eqk T$.
\end{thm}

\begin{cor}  For any two models $\cu M, \cu N$ of a continuous theory $T$,  a regular ultrafilter $\cu D$  saturates $\cu M$ if and only if $\cu D$ saturated $\cu N$.
Thus $\cu D$ saturates $T$ if and only if $\cu D$ saturates some model of $T$.
\end{cor}

It is obvious that the relation $\lek$ is transitive, so it follows that $\lek$ is a pre-ordering on the set of all continuous theories.

\begin{df}  A continuous  theory $T$ is $\lek$-\emph{minimal} if $T\lek U$ for every continuous  theory $U$.
A continuous theory $U$ is $\lek$-\emph{maximal} if $T\lek U$ for every continuous theory $T$.
We define $\Sat(T)=\Sat(\cu M)$ where $\cu M$ is a model of $T$.
\end{df}

Note that any two $\lek$-minimal theories are $\lek$-equivalent, and any two $\lek$-maximal theories are $\lek$-equivalent.

We  recall some definitions from [Ke64].

\begin{df} \label{d-good} Let $\cu D$ be an ultrafilter over $I$.  For any set $J$, let $\cu P_{\aleph_0}(I)$ be the set of finite subsets of $J$.
A mapping $\delta\colon \cu P_{\aleph_0}(J)\to\cu D$ is called
\emph{monotone} if $\delta(u)\supseteq \delta(v)$ whenever $u\subseteq v\in\cu P_{\aleph_0}(J)$, and is called \emph{multiplicative}
if $\delta(u\cup v)=\delta(u)\cap \delta(v)$ whenever $u,v\in \cu P_{\aleph_0}(J)$.  We say that a mapping $\delta'\colon \cu P_{\aleph_0}(I)\to\cu D$
\emph{refines} $\delta$ if $\delta'(u)\subseteq \delta(u)$ for all $u\in\cu P_{\aleph_0}(J)$.  $\cu D$ is said to be \emph{good}
if whenever $|J|\le|I|$, every monotone function $\delta\colon\cu P_{\aleph_o}(J)\to\cu D$ has a multiplicative refinement.
\end{df}

\begin{fact}  \label{f-good} (Kunen [Ku72] in ZFC, [Ke64] under the GCH) For every $I$, there exists a good ultrafilter over $I$.
\end{fact}

\begin{lemma}  \label{l-min-max-G}  ([Ke67] for the first order case).
\begin{itemize}
\item[(i)] There exist $\lek$-minimal theories.  $T$ is $\lek$-minimal if and only if $\Sat(T)$ is the  class of all regular ultrafilters.
\item[(ii)]  There exist $\lek$-maximal theories.  $T$ is $\lek$-maximal if and only if $\Sat(T)$ is the class of all good ultrafilters.
\end{itemize}
\end{lemma}

\begin{proof} (i): By Fact \ref{f-cardinality-ultrapower}, for an infinite set $M$, $|M_\cu D|=|M^I|\ge |I|^+$.  Therefore the
first order theory of an infinite set with only the equality predicate is $\lek$-minimal, so  (i) holds.

(ii):
In [Ke67], it is proved that every $\cu D$ that saturates (for example) $(\cu P_{\aleph_0}(\BN),\subseteq)$ or $(\BN,+,\times)$ is good.
In [Ke64], it is proved (in ZFC) that an ultrafilter $\cu D$ is good if and only if it saturates every first order structure.
 The same proof works for real-valued structures, and (ii) follows.
\end{proof}

\begin{cor}  \label{c-min-max-G}

(i)  A first order theory is $\lek$-minimal if and only it is $\lek$-minimal among first order theories.

(ii)  A first order theory is $\lek$-maximal if and only it is $\lek$-maximal among first order theories.
\end{cor}

\begin{cor}  ([Ke67] for the first order case).
 If $T$ is $\lek$-minimal and $U$ is $\lek$-maximal, then $T\lk U$.
\end{cor}

\begin{proof} If $I$ is uncountable, there is a regular ultrafilter over $I$ that is not good.
\end{proof}

Let us take a brief look at the big picture.
Let $\BF, \BM,$ and $\BC$ be the sets of $\lek$-equivalence classes of first order theories,  metric theories, and continuous theories respectively.
Then each of $(\BF,\lek), (\BM,\lek),$ and $(\BC,\lek)$ is a partial ordering.
By Lemma  \ref{l-min-max-G}, $(\BM,\lek)$ has a $\lek$-minimal element $\min_\BM$ and a $\lek$-maximal element $\max_\BM$, and  $\min_\BM\lk\max_\BM$.
Similarly for $\BC$ and $\BF$.

  Let $\tau\colon\BF\to\BM$ be the mapping
that sends the equivalence class of $Th(\cu M)$ in $\BF$ to the equivalence class of $Th(\cu M)$ in $\BM$ for every first order structure $\cu M$.

\begin{cor} \label{c-tau} $\tau$ is an isomorphism from $(\BF,\lek)$ onto a substructure of $(\BM,\lek)$.  Moreover, $\tau(\min_\BF)=\min_\BM$, and $\tau(\max_\BF)=\max_\BM$.
\end{cor}

\begin{proof}  By Corollary \ref{c-min-max-G}.
\end{proof}

We have not been able to answer the following important question.

\begin{question} \label{q-1} Is every metric theory $\lek$-equivalent to a first order theory?
Equivalently, is the mapping $\tau\colon\BF\to\BM$ onto?
\end{question}

The answer is ``no'' if and only if there are ``new''  $\lek$-equivalence classes of metric theories that do not correspond to
$\lek$-equivalence classes of first order theories.

We will see later, in Theorem \ref{t-metric-expansion}, that  every $\lek$-equivalence class in $\BC$ contains a metric theory.
It follows that the mapping that sends the equivalence class of each metric theory in $\BM$ its equivalence class in $\BC$ is an isomorphism from $(\BM,\lek)$ onto $(\BC,\lek)$.
So we regard the partial ordering $(\BC,\lek)$ as the same as $(\BM,\lek)$.
Beginning in Section \ref{s-stable}, we will investigate the structure of the partial ordering $(\BM,\lek)$.

One can also turn things around and use theories to compare ultrafilters instead of using ultraproducts to compare theories.

\begin{df}   $\Sat_\BM(\cu D)$ is the set of $H\in\BM$ such that $\cu D$ saturates some, or equivalently every, $T\in H$.
$\Sat_\BF(\cu D)$ is defined similarly with $\BF$ in place of $\BM$.
\end{df}

\begin{cor} \label{c-tau-min-max} (i) For all $H\in\BF$,  $H\in\Sat_\BF(\cu D)\Leftrightarrow\tau(H)\in\Sat_\BM(\cu D)$.

(ii) $\{\cu D\colon \Sat_\BF(\cu D)=\BF\}=\{\cu D\colon \Sat_\BM(\cu D)=\BM\}=\{\cu D\colon\cu D \mbox{ is good}\}.$
\end{cor}

\begin{proof}  (i) is trivial.  (ii)  follows from Lemma \ref{l-min-max-G} and Corollary \ref{c-min-max-G}.
\end{proof}

We now review some  results of Malliaris and Shelah about first order theories.
In  $(\BF,\lek)$, the lowest three equivalence classes are, in increasing order,
the minimal class, the class of stable non-minimal first order theories, and the equivalence class $rg$ of the theory $T_{rg}$ of the random graph.  $rg$ is the minimal
class of  unstable first order theories.  The maximal class contains the class of SOP$_2$ first order theories, and
Malliaris and Shelah conjecture that the maximal class is exactly the class of SOP$_2$ first order theories.
An first order theory is simple if and only if it has neither SOP$_2$ nor TP$_2$ (and $T_{rg}$ is simple).
There are continuum many
distinct equivalence classes of simple first order theories.  There is  an first order theory $T^*_{feq}$ that is minimal among TP$_2$ theories,
and hence minimal among non-simple first order theories.  Also, $T_{rg}\lk T^*_{feq}$.
It is open whether or not the $T^*_{feq}$ is maximal.

As mentioned in the Introduction, these first order results show that the partial ordering $(\BF,\lek)$ looks like Figure 1.
In this paper we will show that the same picture applies to the partial ordering $(\BM,\lek)$.
To do so, we will generalize several notions and results from first order model theory to continuous model theory.

\section{Elementary equivalent structures are equally saturatable}

In this section we will prove Theorem \ref{t-equiv-triangle} above.
The proof will be  longer than the proof of the corresponding first order result in [Ke67],
because  continuous logic does not have negation, and the $\inf$ quantifier may only have approximate witnesses in a real-valued structure.
We will first prove two lemmas, and
use those lemmas to prove a stronger version of Theorem \ref{t-equiv-triangle} that concerns  ultraproducts,
rather than  ultrapowers.

In what follows, $\cu M, \cu N$ will be  $V$-structures.

\begin{lemma}  \label{l-onestep}
Suppose $\cu M\equiv \cu N$.  Let $\Psi(\vec v)$ be a finite set of  formulas with the free variables $\vec v$.
For every $n\in\BN$ and every $|\vec v|$-tuple $\vec a$ in $\cu M$  there is a $|\vec v|$-tuple $\vec b$ in $\cu N$ such that
for every $\psi\in\Psi$,
$\psi^{\cu M}(\vec a)$ is within $2^{-n}$ of $\psi^{\cu N}(\vec b).$
\end{lemma}

\begin{proof}
 Let $\vec a\in\cu M^{|\vec v|}$.
For each $\psi\in\Psi$, let $r_\psi$ be a dyadic rational that is within $2^{-(n+1)}$ of $\psi^{\cu M}(\vec a)$.
Since $\Psi$ is finite,
$$\max_{\psi\in\Psi}|\psi(\vec v)-r_\psi|\dotminus 2^{-(n+2)}$$
is a  formula that is satisfied by $\vec a$ in $\cu M$.  Then by Remark \ref{r-dotminus},
$$\cu M\models(\inf_{\vec v})\max_{\psi\in\Psi}|\psi(\vec v)-r_\psi|\le 2^{-(n+2)},$$
and since $\cu M\equiv\cu N$,
$$\cu N\models(\inf_{\vec v})\max_{\psi\in\Psi}|\psi(\vec v)-r_\psi|\le 2^{-(n+2)}.$$
Therefore there is a $|\vec v|$-tuple $\vec b$ in $\cu N$ such that
$$\cu N\models\max_{\psi\in\Psi}|\psi(\vec b)-r_\psi|\le2^{-(n+1)},$$
and hence
$$|\psi^{\cu M}(\vec a)-\psi^{\cu N}(\vec b)|\le 2^{-n}.$$
for every $\psi\in\Psi$.
\end{proof}

\begin{cor}  \label{c-onestep}
Suppose $\cu M\equiv \cu N$, and $\cu N$ is $\aleph_1$-saturated.  Then for every $\vec a\in\cu M^{|\vec v|}$ there exists $\vec b\in \cu N^{|\vec v|}$ such that
$(\cu N,\vec b)\equiv(\cu M,\vec a)$.
\end{cor}

\begin{lemma}  \label{l-twostep}
Suppose $\cu M\equiv \cu N$.  Let $\Psi(\vec u,\vec v)$ be a finite set of  formulas with the free variables $\vec u,\vec v$.
For every $n\in\BN$ and every $|\vec v|$-tuple $\vec a$ in $\cu M$  there is a $|\vec v|$-tuple $\vec b$ in $\cu N$ such that
\begin{itemize}
\item[(i)] For every $\vec c\in\cu M^{|\vec u|}$ there exists $\vec e\in \cu N^{|\vec u|}$ such that for every $\psi\in\Psi$,
$\psi^{\cu M}(\vec c,\vec a)$ is within $2^{-n}$ of $\psi^{\cu N}(\vec e,\vec b)$.
\item[(ii)] For every $\vec e\in\cu N^{|\vec u|}$ there exists $\vec c\in \cu M^{|\vec u|}$ such that for every $\psi\in\Psi$,
$\psi^{\cu M}(\vec c,\vec a)$ is within $2^{-n}$ of $\psi^{\cu N}(\vec e,\vec b)$.
\end{itemize}
\end{lemma}

\begin{proof}  Let $(z_1,\ldots,z_k)$ be a tuple of disjoint $|\vec u|$-tuples of variables that are also disjoint from $\vec u$ and $\vec v$.
In $\cu M$, we say that an assignment  of $(z_1,\ldots,z_k)$ is $2^{-n}$-\emph{dense over} $\vec a$ if for every $|\vec u|$-tuple $\vec c$ there is an $\ell\le k$
such that for every $\psi\in\Psi$, $\psi(\vec c,\vec a)$ is within $2^{-n}$ of $\psi(z_\ell,\vec a)$.
It is clear that in $\cu M$, for any $\vec a$ and any $n\in\BN$, there exists an assignment $(z_1,\ldots,z_k)$ that is $2^{-n}$-dense over $\vec a$.

Let $\theta(z_1,\ldots,z_k,\vec v)$ be the formula
$$ (\sup_{\vec u})\min_{\ell\le k}\max_{\psi\in\Psi}(|\psi(\vec u,\vec v)-\psi(z_\ell,\vec v)|).$$
Note that in $\cu M$, $(z_1,\ldots,z_k)$  is $2^{-n}$-dense over $\vec a$ if and only if
$$ \cu M\models \theta(z_1,\ldots,z_k,\vec a)\le 2^{-n}.$$

Now fix a $|\vec v|$-tuple $\vec a$ in $\cu M$.
Suppose $z=(z_1,\ldots,z_k)$ is $2^{-(n+2)}$-dense over $\vec a$ in $\cu M$.
Since $\cu N\equiv\cu M$ and $\theta$ is a formula, it follows from Lemma \ref{l-onestep} that there exist tuples $(w_1,\ldots,w_k,\vec b)$ in $\cu N$ such that
\begin{itemize}
\item[(1)] $(w_1,\ldots,w_k)$ is $2^{-(n+1)}$-dense over $\vec b$ in $\cu N$.
\item[(2)] For each $\psi\in\Psi$ and $1\le\ell\le k$,
$\psi^{\cu N}(w_\ell,\vec b)$ is within $2^{-(n+1)}$ of $\psi^{\cu M}(z_\ell,\vec a)$.
\end{itemize}

Proof of (i):  Let $\vec c\in\cu M^{|\vec u|}$.  By (1), there is an $\ell\in\{1,\ldots,k\}$ such that for every $\psi\in\Psi$,
$\psi^{\cu M}(z_\ell,\vec a)$ is within $2^{-(n+2)}$ of $\psi^{\cu M}(\vec c,\vec a)$.  Take $\vec e=z_\ell$.  By (2), for each $\psi\in\Psi$,
$\psi^{\cu N}(\vec e,\vec b)$ is within $2^{-(n+1)}$ of $\psi^{\cu M}(z_\ell,\vec a)$.  Therefore $\psi^{\cu N}(\vec e,\vec b)$ is within $2^{-n}$ of $\psi^{\cu M}(\vec c,\vec a)$, as required.

Proof of (ii):  Let $\vec e\in\cu N^{|\vec u|}$.  By (1), there is an $\ell\in\{1,\ldots,k\}$ such that for every $\psi\in\Psi$,
$\psi^{\cu N}(w_\ell,\vec b)$ is within $2^{-(n+1)}$ of $\psi^{\cu N}(\vec e,\vec b)$.  Take $\vec c=z_\ell$.  By (2), for each $\psi\in\Psi$,
$\psi^{\cu M}(\vec c,\vec a)$ is within $2^{-(n+1)}$ of $\psi^{\cu N}(w_\ell,\vec b)$.  Therefore $\psi^{\cu M}(\vec c,\vec a)$ is within $2^{-n}$ of $\psi^{\cu N}(\vec e,\vec b)$.
\end{proof}

The next result, Theorem \ref{t-lambda-triangle-ultraproduct}, is the stronger form of Theorem \ref{t-equiv-triangle} that we promised.

\begin{thm}  \label{t-lambda-triangle-ultraproduct}
Suppose $\cu D$ is a regular ultrafilter over a set $I$ of infinite cardinality $\lambda$, $\{\cu M_t\colon t\in I\}$ and $\{\cu N_t\colon t\in I\}$
are families of real-valued structures with the same vocabulary, and $\cu M_t\equiv\cu N_t$ for each $t\in I$.  Then
 $\prod^{\cu D}\cu M_t$ is $\lambda^+$-saturated if and only if
$\prod^{\cu D}\cu N_t$ is $\lambda^+$-saturated.
\end{thm}

Our proof of Theorem \ref{t-lambda-triangle-ultraproduct} will use a continuous analogue of the notion of a distribution from Malliaris [Ma12].
Distributions give a useful criterion for a set of formulas with parameters to be satisfiable in a regular ultraproduct.
In what follows, we let  $\cu M'=\prod^\cu D \cu M_t$ be a pre-ultraproduct, $A$ be a set of parameters in $\cu M'$
of cardinality $|A|\le \lambda$, and $\Gamma=\Gamma(\vec x,A)$ be a set of at most $\lambda$ continuous formulas in the parameters $A$.  For each $t\in I$ and $a\in A$,
$a[t]$ will be the corresponding element of $\cu M_t$, and $A[t]=\{a[t]\colon a\in A\}$.

\begin{df}  We define the approximation $\Gamma^{ap}$ of $\Gamma$ to be the set of formulas
$$\Gamma^{ap}=\Gamma^{ap}(\vec x,A)=\{\psi(\vec x,A)\dotminus 2^{-n}\colon \psi(\vec x,A)\in\Gamma, n\in\BN\}.$$
\end{df}

\begin{rmk}  \label{r-approx-sat}   A tuple $\vec c$ satisfies $\Gamma$ in $\cu M'$ if and only
if $\vec c$ satisfies $\Gamma^{ap}$ in $\cu M'$. $\Gamma$ is finitely satisfiable in $\cu M'$ if and only if
$\Gamma^{ap}$ is finitely satisfiable in $\cu M'$.
\end{rmk}

\begin{proof}   By Fact \ref{r-regular-aleph-1},  $\cu M'$ is $\aleph_1$-saturated.
\end{proof}

\begin{df}  \label{d-distribution}
A \emph{distribution} of $\Gamma$ in $\cu M'$ is a monotone mapping $\delta\colon\cu P_{\aleph_0}(\Gamma^{ap})\to \cu D$ such that:
\begin{itemize}
\item[(a)] $\delta\colon\cu P_{\aleph_0}(\Gamma^{ap})\to \cu X$ for some $\cu X$ that regularizes $\cu D$.
\item[(b)]  For each $\Psi\in\cu P_{\aleph_0}(\Gamma^{ap})$ and $t\in \delta(\Psi)$,
$$\cu M_t\models(\inf_{\vec x})\max_{\psi\in\Psi}\psi(\vec x,A[t])=0.$$
\end{itemize}
For each $t\in I$, we define
$$\Gamma(\delta,t)=\{\psi\in\Gamma^{ap}\colon t\in \delta(\{\psi\})\}.$$

A distribution $\delta$ of $\Gamma$ in $\cu M'$ is \emph{accurate} if
\begin{itemize}
\item[(c)] For all $\Psi\in\cu P_{\aleph_0}(\Gamma^{ap})$,
$$\delta(\Psi)=\left\{t\in\left(\bigcap_{\psi\in\Psi} \delta(\{\psi\})\right)\colon \cu M_t\models(\inf_{\vec x})\max_{\psi\in\Psi}\psi(\vec x,A[t])=0\right\}.$$
\end{itemize}

\end{df}

If $\cu M$ is a first order structure and $\Gamma$ is a set of first order formulas, the notion of a distribution is the same as above except that
$\Gamma^{ap}$ is replaced by $\Gamma$, and $(\inf_{\vec x})\max_{\psi\in\Psi}\psi(\vec x,A[t])=0$ is replaced by
$(\exists\vec x)\bigwedge_{\psi\in\Psi}\psi(\vec x,A[t])$.

\begin{rmk}  \label{r-distribution}  Let $\delta$ be a distribution of $\Gamma$ in $\cu M'$, and let $\Psi\in\cu P_{\aleph_0}(\Gamma^{ap})$.

(i)  For each $t\in I$, $\Gamma(\delta,t)$ is finite.

(ii)  $\delta(\Psi)\subseteq \{t\in I\colon \Psi\subseteq \Gamma(\delta,t)\}\in\cu D$.

(iii)  If $\delta$ is multiplicative, then $\delta(\Psi)= \{t\in I\colon \Psi\subseteq \Gamma(\delta,t)\}$ and $\delta$ is accurate.

(iv) We think of $\Gamma(\delta,t)$ as the part of $\Gamma^{ap}$ that matters at the index $t$.
\end{rmk}

\begin{lemma}  \label{l-dist-exist}  The following are equivalent.
\begin{itemize}
\item[(i)] $\Gamma$ is finitely satisfiable in $\cu M'$.
\item[(ii)] $\Gamma$ has a distribution in $\cu M'$.
\item[(iii)] $\Gamma$ has an accurate distribution in $\cu M'$.
\end{itemize}
\end{lemma}

\begin{proof} (ii) $\Rightarrow$ (i): By Remark \ref{r-regular-aleph-1}, $\cu M'$ is $\aleph_1$-saturated, so by Remark \ref{r-approx-sat}, $\Gamma$ is finitely satisfiable
in $\cu M'$ if and only if $\Gamma^{ap}$ is finitely satisfiable in $\cu M'$.  If $\delta$ is a distribution of $\Gamma$ in $\cu M'$
and $\Psi$ is a finite subset of $\Gamma^{ap}$, then  $\delta(\Psi)\in\cu D$, so by (c), the \Los \ Theorem, and the $\aleph_1$-saturation of $\cu M'$, $\Psi$ is satisfiable
in $\cu M'$.

(i) $\Rightarrow$ (iii): Suppose $\Gamma^{ap}$ is finitely satisfiable in $\cu M'$.  Since $\cu D$ is regular, $\cu D$ has a regularizing set $\cu X$.  Thus $\cu X\subseteq\cu D$,
$|\cu X|=\lambda$, and each $t\in I$ belongs to only finitely many $X\in\cu X$.  So there is an injective function $h\colon \Gamma^{ap}\to\cu X$.
For $\psi\in\Gamma^{ap}$ let $\delta(\{\psi\})=h(\psi)$.  Then
the mapping $\delta$ defined by condition (c) of Definition \ref{d-distribution}
is an accurate distribution of $\Gamma$ in $\cu M'$.
\end{proof}

Since $\cu D$ is regular and $I$ is infinite, we may choose sets $J_1\supseteq J_2\supseteq\cdots$ in $\cu D$ such that $\bigcap_n J_n$ is empty.
Put $J_0=I$, and for each $t\in I$, let $n(t)$ be the greatest $n\in\BN$ such that $t\in J_n$.

\begin{lemma} \label{l-dist-saturates} The following are equivalent.
\begin{itemize}
\item[(i)]   $\Gamma$ is satisfiable in $\cu M'$.
\item[(ii)] $\Gamma$ has a multiplicative distribution in $\cu M'$.
\item[(iii)] $\Gamma$ is finitely satisfiable in $\cu M'$ and every accurate distribution of $\Gamma$ in $\cu M'$ has a multiplicative refinement.
\end{itemize}
\end{lemma}

\begin{proof}  (i) $\Rightarrow$ (ii):
Suppose $\vec c$ satisfies $\Gamma$ in $\cu M'$.  Let $\cu X$ regularize $\cu D$ and let  $h\colon \Gamma^{ap}\to\cu X$
be injective.  For each $\Psi\in\cu P_{\aleph_0}(\Gamma^{ab})$, let
$$\delta(\Psi)=\bigcap_{\psi\in\Psi}(h(\psi)\cap\{t\in I\colon \cu M_t\models \psi(\vec c[t],A[t])=0\}).$$
It is clear that $\delta$ is multiplicative, and each $t\in I$ belongs to $\delta(\Psi)$ for only finitely many $\Psi$.  Consider a formula $\psi(\vec x,A)\in\Gamma^{ap}$.
For some $\theta\in\Gamma$ and $n\in\BN$, we have $\psi(\vec x,A)=(\theta(\vec a,A)\dotminus 2^{-n})$,
and $\cu M'\models \theta(\vec c,A)$.  By the \Los \ Theorem,
$$\{t\in I\colon \cu M_t\models \psi(\vec c[t],A[t])=0\}\in\cu D.$$
Therefore $\delta$ maps $\cu P_{\aleph_0}(\Gamma^{ap})$ into $\cu D$, and hence $\delta$ is a multiplicative distribution of $\Gamma$ in $\cu M'$.

(ii) $\Rightarrow$ (i):  Now suppose $\delta$ is a multiplicative distribution of $\Gamma$ in $\cu M'$.  Let $t\in I$.  By (a),  the set
$$\Psi=\{\psi\in\Gamma^{ap}\colon t\in \delta(\{\psi\})\}$$
belongs to $\cu P_{\aleph_0}(\Gamma^{ap})$. By multiplicity,
$\delta(\Psi)=\bigcap_{\psi\in\Psi} \delta(\{\psi\})$, so $t\in \delta(\Psi)\in\cu D.$
By \ref{d-distribution} (b), there is a tuple $\vec c[t]$ in $\cu M_t$ such that
$$ (\forall \psi\in\Psi) \cu M_t\models  \psi(\vec c[t],A[t])\le 2^{-n(t)}.$$
Then for each $\Psi\in\cu P_{\aleph_0}(\Gamma^{ab})$ and $n\in\BN$, we have
$$X:=\delta(\Psi)\cap J_n\in\cu D, \quad (\forall t\in X)(\forall \psi\in\Psi) \cu M_t\models  \psi(\vec c[t],A[t])\le 2^{-n}.$$
By the \Los \ Theorem, for each $\psi\in\Gamma^{ab}$ and $n\in\BN$ we have $\cu M'\models \psi(\vec c,A)\le 2^{-n}$, so $\cu M'\models \psi(\vec c,A)=0$.

(iii) $\Rightarrow$ (ii):  Assume (iii).  By Lemma \ref{l-dist-exist}, $\Gamma$ has an accurate distribution in $\cu M'$, so by  (iii),
$\Gamma$ has a multiplicative distribution in $\cu M'$.

(i) $\Rightarrow$ (iii):  Suppose $\vec c$ satisfies $\Gamma$ in $\cu M'$ and $\delta$ is an accurate distribution of $\Gamma$ in $\cu M'$.
  For $\Psi\in\cu P_{\aleph_0}(\Gamma^{ap})$, let
$$\delta'(\Psi)=\left\{t\in\delta(\Psi)\colon \cu M_t\models\max_{\psi\in\Psi} \psi(\vec c[t],A[t])=0)\right\}.$$
Then $\delta'$ is a distribution of $\Gamma$ in $\cu M'$ that refines $\delta$.
Suppose  $t\in\delta'(\Psi)\cap\delta'(\Theta)$. Then
$$ t\in\delta(\Psi)\cap\delta(\Theta) \mbox{ and } \cu M_t\models\max_{\psi\in\Psi\cup\theta} \psi(\vec c[t],A[t])=0).$$
Since $\delta$ is accurate, $t\in\delta(\Psi\cup\Theta)$, so $t\in\delta'(\Psi\cup\Theta)$ and $\delta'$ is multiplicative.
\end{proof}

\begin{proof}  [Proof of Theorem \ref{t-lambda-triangle-ultraproduct}]
Let $\cu M'=\prod^{\cu D}\cu M_t$ and $\cu N'=\prod^{\cu D}\cu N_t$.
We suppose that $\cu M'$ is $\lambda^+$-saturated, and prove that $\cu N'$ is $\lambda^+$-saturated.  Let $B$ be a $\lambda$-sequence
of elements of $\cu N'$, and assume that $\Gamma(u,B)$ is a set of at most $\lambda$ formulas  with parameters from $B$ and at most $u$ free that is
is finitely satisfiable in $\cu N'$.
By Lemma \ref{l-strict-sat}, to show that $\cu N'$ is $\lambda^+$-saturated it
suffices to prove that  $\Gamma(u,B)$ is satisfiable in $\cu N'$.
By Lemma \ref{l-dist-exist}, $\Gamma(u,B)$ has a distribution $\delta_1$ in $\cu N'$.
For each $\psi\in\Gamma^{ap}(u,B)$, we have $\psi=(\theta(\psi)\dotminus 2^{-m(\psi)})$ for some $\theta(\psi)\in\Gamma(u,A)$ and $m(\psi)\in\BN$,
and we define $\psi^+=(\theta(\psi)\dotminus 2^{-(m(\psi)+1)})$.
Then $\psi\mapsto \psi^+$ is a bijection from $\Gamma^{ap}(u,B)$ onto itself.  Define   $\Psi^+=\{\psi^+\colon \psi\in\Psi\}$.

For each $t\in I$, let $m(t)=\max\{m(\psi)\colon \psi\in\Gamma(\delta_1,t)\}$.
By Lemma \ref{l-twostep}, there is a $\lambda$-sequence $A$ of elements of $\cu M'$ such that for every $t\in I$,
\begin{itemize}
\item[(1)] For every $e[t]\in\cu N_t$ there exists $c[t]\in\cu M_t$ such that for every $\psi$ in $\Gamma(\delta_1,t)\}$,
$\theta(\psi)^{\cu N_t}(e[t],B[t])$ is within $2^{-(m(t)+1)}$ of $\theta(\psi)^{\cu M_t}(c[t],A[t])$.
\item[(2)] For every $c[t]\in\cu M_t$ there exists $e[t]\in\cu N_t$ such that for every $\psi$ in $\Gamma(\delta_1,t)\}$,
$\theta(\psi)^{\cu N_t}(e[t],B[t])$ is within $2^{-(m(t)+1)}$ of $\theta(\psi)^{\cu M_t}(c[t],A[t])$.
\end{itemize}
\medskip

It follows from (1) that the mapping $\delta_2\colon\cu P_{\aleph_0}(\Gamma^{ap}(u,A))\to \cu D$ such that
$$\delta_2(\Psi(u,A))=\delta_1(\Psi^+(u,B))$$
is a distribution of $\Gamma(u,A)$ in $\cu M'$, so by Lemma \ref{l-dist-exist}, $\Gamma(u,A)$ is finitely satisfiable in $\cu M'$.
Since $\cu M'$ is $\lambda^+$-saturated, $\Gamma(u,A)$ is satisfiable in $\cu M'$.
Then by Lemma \ref{l-dist-saturates}, there is a multiplicative distribution $\delta_3$ of $\Gamma(u,A)$ in $\cu M'$.
It follows from (2) that
the mapping $\delta_4\colon\cu P_{\aleph_0}(\Gamma^{ap}(u,B))\to \cu D$ such that
$$\delta_4(\Psi(u,B))=\delta_3(\Psi^+(u,A))$$
is a multiplicative distribution of $\Gamma(u,B)$ in $\cu N'$, so by Lemma \ref{l-dist-saturates}, $\Gamma(u,B)$ is  satisfiable in $\cu N'$.
\end{proof}

\section{Every continuous theory is $\lek$-equivalent to a metric theory}  \label{s-metric}

In this short section we will prove the following.

\begin{thm}  \label{t-metric-expansion}  For every continuous theory $T$ there us a metric theory $U$ (which may have a different vocabulary) such that $T \eqk U$.
\end{thm}

We will see that Theorem \ref{t-metric-expansion} is an easy consequence of Theorem \ref{t-equiv-triangle} above and
results about premetric expansions from the paper [Ke21] (Fact \ref{f-expansion}--\ref{f-expansion-sat} below).

We say that a sequence $\langle\varphi_m(\vec x,\vec y)\rangle_{m\in\BN}$ of formulas is \emph{Cauchy} in a $V$-structure $\cu M$
if for each $\varepsilon >0$ there exists $m$ such that for all $k\ge m$,
$$ \cu M\models (\sup_{\vec x})(\sup_{\vec y})|\varphi_m(\vec x,\vec y)-\varphi_k(\vec x,\vec y)|\le\varepsilon.$$
If $\langle\varphi_m(\vec x,\vec y)\rangle_{m\in\BN}$ is Cauchy in $\cu M$, then there is a
unique mapping from $M^{|\vec x|}\times M^{|\vec y|}$ into $[0,1]$, denoted by $[\lim \varphi_m]^{\cu{M}}$, such that
$$(\forall \vec b\in M^{|\vec x|}) (\forall \vec c\in M^{|\vec y|})[\lim \varphi_m]^\cu{M}(\vec b,\vec c)=\lim_{m\to\infty}\varphi_m^{\cu{M}}(\vec b,\vec c).$$

\begin{df} \label{d-metric-expansion}
 Let $T$ be a  $V$-theory,  let $D$ be a  predicate symbol
that may or may not belong to $V$, and let $V_D=V\cup\{ D\}$.
We say that a $V_D$-theory $T_e$ is a \emph{premetric expansion} of $T$  if:
\begin{itemize}
\item[(i)]  There is  a metric signature $L_e$ over $V_D$ with distance predicate $D$ such that
 $T_e$ is an $L_e$-metric theory.
\item[(ii)] There is a Cauchy sequence $\langle d\rangle=\langle d_m\rangle$ of $V$-formulas  such that
the models of $T_e$ are exactly the $V_D$-structures
of the form $\cu M_e=(\cu M,[\lim d_m]^{\cu M})$, where $\cu M$ is a  model of $T$.
\end{itemize}
It follows that $\cu M_e$ is an $L_e$-premetric structure, which we call  the  \emph{premetric expansion} of  $\cu M$  to $T_e$.
\end{df}

\begin{fact}  \label{f-expansion}  (Expansion Theorem, Theorem 3.3.4 of [Ke21])  Every continuous  theory $T$ has a premetric expansion.
\end{fact}

\begin{fact}  \label{f-expansion-unique}  (By Propositions 3.4.5 and 4.3.4 of [Ke21])
Every reduced $V$-structure $\cu M$ has a unique topology that is metrizable by a metric $D$ (not unique)
such that $(\cu M,D)$ is a premetric expansion of $\cu M$.
\end{fact}

The following fact can be used to show that when $T_e$ is a premetric expansion of $T$, various properties hold
for $T$ if and only if they hold for $T_e$.

\begin{fact}  \label{f-expansion-approx}  (Lemma 4.2.1 in [Ke21])  Suppose $T_e$ is a premetric expansion of $T$.  Then for every continuous formula
$\varphi(\vec x)$ in the vocabulary of $T_e$, and every real $\varepsilon>0$, there is a formula $\theta(\vec x)$ in the vocabulary of $T$
such that
$$T_e\models(\sup_{\vec x})|\varphi(\vec x)-\theta(\vec x)|\le\varepsilon.$$
\end{fact}

\begin{fact} \label{f-expansion-product} (Proposition 3.1.6 of [Ke21]).  Suppose $T_e$ is a premetric expansion of $T$.
Then for every pre-ultraproduct $\prod^{\cu D}\cu M_t$ of
models of $T$, $(\prod^{\cu D}\cu M_t)_e = \prod^{\cu D}((\cu M_t)_e)$.
\end{fact}

\begin{fact} \label{f-expansion-sat} (Proposition 4.1.6 of [Ke21])
If $\cu M_e$ is a premetric expansion of $\cu M$, then for each infinite cardinal $\kappa$,
$\cu M_e$ is $\kappa$-saturated if and only if $\cu M$ is $\kappa$-saturated.
 \end{fact}

\begin{thm}  \label{t-expansion-equiv}
 If $T$ is a continuous theory and $T_e$ is a premetric expansion of $T$,
then $T_e\eqk T$.
\end{thm}

\begin{proof} [Proof of Theorem \ref{t-expansion-equiv}]   Let $\cu M$ be a model of $T$.
By Fact \ref{f-expansion-product}, $ \prod^{\cu D}(\cu M_e)=(\prod^{\cu D}\cu M)_e$, so $\prod^{\cu D}(\cu M_e)$ is a premetric expansion of $\prod^{\cu D}\cu M$.
By Fact \ref{f-expansion-sat}, $\prod^{\cu D}(\cu M_e)$ is $\lambda^+$-saturated if and only if $\prod^{\cu D}\cu M$  is $\lambda^+$-saturated.
Therefore by Theorem \ref{t-equiv-triangle}, we have $T_e\eqk T$.
\end{proof}

\begin{proof}  [Proof of Theorem \ref{t-metric-expansion}]  By Fact \ref{f-expansion}, $T$ has a pre-metric expansion $T_e$.
By definition, $T_e$ is a metric theory.  By Theorem \ref{t-expansion-equiv}, $T_e\eqk T$.
\end{proof}

\section{Stable theories}  \label{s-stable}

In view of the Theorem \ref{t-expansion-equiv}, from here on we may confine our attention to metric theories.

\begin{convention} For the rest of this paper, $L$ will be a metric signature over $V$ with distance predicate $d\in V$, and $T$  will be an $L$-metric theory.
\end{convention}

Thus by Corollary \ref{c-model-premetric}, every model of $T$ is an $L$-premetric structure.

In this section we study the $\lek$ ordering on stable metric theories.
Theorem \ref{t-stable-vs-unstable} (iii) below will show that if $T, U$ are metric theories, $T$ is stable, and $U$ is unstable, then $T\lk U$.

The papers [BBHU08] and [BU09] gave equivalent definitions of a stable  metric theory.    Here we will use the definition in [BU09]
(Definition \ref{d-stable-local} below).

For every result proved in this section, the corresponding result
for first order theories was proved by Shelah in Section VI of [Sh78].  The arguments in this section are similar to the arguments in [Sh78],
but generalized  to the continuous case.
Many first order results about  Morley sequences and indiscernible sets have been extended to continuous model logic (e.g. in [BU09], [BY03], [BY09], [BY13], and [EG12]),
but we will need to extend some additional results about indiscernible sets from first order logic to continuous logic here.
 It will be convenient to use
the properties of a stable independence relation (Definition \ref{d-stable}) rather than the notion of a non-forking extension of a type,
and sometimes to use ultraproducts of real-valued structures rather than just metric structures.

\begin{df} \label{d-stable-local} (Local stability.  See Definition 7.1 of [BU09]).

\noindent(i)  A continuous formula $\varphi(\vec x,\vec y)$ is \emph{unstable with bounds $(s,r)$ in $T$}
if  $0\le s<r\le 1$ and in some model
$\cu M\models T$ there is an infinite sequence $\langle \vec a_h,\vec b_h\rangle_{h\in\BN}$  such that whenever $h<k$,
$\varphi^{\cu M}(\vec a_h,\vec b_k)\le s$ and $\varphi^{\cu M}(\vec a_k,\vec b_h)\ge r$.  $\varphi$ is \emph{unstable} in $T$ if it is unstable in $T$ for some $(s,r)$,
and is \emph{stable} in $T$ otherwise.

(ii) $T$ is stable if  every continuous formula $\varphi(\vec x,\vec y)$ is stable in $T$.
\end{df}

Note that if $\varphi(\vec x,\vec y)$ is unstable in $T$, then there is  a model $\cu M$ as in (i) that  has a countable dense set.

\begin{cor}  \label{c-local-unstable}
$T$ is unstable if and only if there is a continuous formula $\theta(\vec x,\vec y)$ that is unstable for $(0,1)$ in $T$.
That is, in some model
$\cu M\models T$ there is an infinite sequence $\langle \vec a_h,\vec b_h\rangle_{h\in\BN}$  such that whenever $h<k$,
$\varphi^{\cu M}(\vec a_h,\vec b_k)=0$ and $\varphi^{\cu M}(\vec a_k,\vec b_h)=1$.
\end{cor}

\begin{proof}  Suppose $\varphi(\vec x,\vec y)$ is unstable for $(s,r)$ in $T$.
There is a continuous function $f\colon [0,1]\to[0,1]$ that maps $[0,s]$ to $0$ and maps $[r,1]$ to $1$.
Then the formula $f(\varphi(\vec x,\vec y))$ is unstable for $(0,1)$ in $T$.
\end{proof}

As is often done in the literature, we assume  for convenience that there is an uncountable inaccessible cardinal $\Upsilon$.
That assumption simplifies things but can be avoided.
We call a set \emph{small} if it has cardinality less than $\Upsilon$.
By a \emph{monster structure} we mean a   $\Upsilon$-saturated metric structure  of cardinality $\le\Upsilon$.

\begin{fact} \label{f-monster}  Up to isomorphism,  $T$ has a unique monster model $\FM$.
Every small model of $T$  is elementarily embeddable in $\FM$.
\end{fact}

\begin{proof}  The proof that $T$ a unique reduced $\Upsilon$-saturated model $\FM$ of cardinality $\le\Upsilon$  is similar to the proof of the analogous result in first order model theory.
Since $\FM$ is reduced and $\aleph_1$-saturated, it is a metric structure by Remark \ref{r-sat-metric}.
\end{proof}

We recall some notation from the literature.
We sometimes write $AB$ for $A\cup B$.  We write $A\equiv_B A'$ if $\FM_{BA}\equiv \FM_{BA'}$. Given an $n$-tuple $\vec a$ in the monster model $\FM$ of $T$ and a subset $B$ of $\FM$,
the \emph{type} of $\vec a$ over $B$  is the set
$\tp(\vec a/B)$ of all formulas $\varphi(\vec x,B)$ with parameters in $B$ such that $\FM\models \varphi(\vec a,B)=0.$
The set of all $n$-types over $B$ is $S_n(B)=\{\tp(\vec a/B)\colon \vec a\in\FM^n\}$.
For each $p\in S_n(B)$ and formula $\psi(\vec x,B)$ with $|\vec x|=n$, $p^\psi$ is the unique $r\in[0,1]$ such that $|\psi- r|\in p$.

The following definition is equivalent to Definition 14.13 of [BBHU08], in view of Theorem 8.10 of [BU09].

\begin{df} \label{d-stable}
A \emph{stable independence relation} for $T$ is a ternary
relation $\ind$ on the small subsets of the monster model $\FM$ of $T$ that has the following
properties:
\begin{itemize}
\item \emph{Invariance under automorphisms of $\FM$.}
\item \emph{Symmetry:} $A\ind_C B$ if and only if $B\ind_C A$.
\item \emph{Transitivity:} $A\ind_C BD$ if and only if $A\ind_C B$ and $A\ind_{BC} D$.
\item \emph{Finite Character:} $A\ind_C B$ if and only if $\vec a\ind_C B$ for all finite $\vec a\subseteq A$.
\item \emph{Existence:}  For all $A,B,C$ there exists $A'$ such that
$A'\equiv_C A$  and $A'\ind_C B$.
\item \emph{Strong Local Character:} For each finite $\vec a$, there exists $C\subseteq B$ of cardinality $\le \aleph_0$ with $\vec a\ind_{C} B$.
\item \emph{Stationarity:} For all small $B$, algebraically closed $C$,   and tuples $\vec a,\vec a\,'$,
$$ \mbox{If}\quad \vec a\equiv_C \vec a',  \quad \vec a\ind_{C} B, \quad \vec a\,'\ind_{C} B, \quad \mbox{then}\quad  \vec a\equiv_{BC}\vec a'.$$
\end{itemize}
\end{df}

\begin{fact}  \label{f-stable}  (See Theorems 14.6 and 14.14 of [BBHU08].)
$T$ is stable if and only if there is a stable independence relation for $T$, and also if and only if there is a unique stable independence relation for $T$.
\end{fact}

\begin{convention} Hereafter, when $T$ is a stable theory, $\ind$ will denote the unique stable independence relation for $T$.
\end{convention}

By a proper tuple we mean a tuple of distinct elements.
Let $n\in\BN$.  A set $\cu I\subseteq\FM^n$ is called $B$-indiscernible if every two proper tuples of elements of $\cu I$ of the same length
have the same type over $B$.  A sequence $\cu I=\langle \vec a_h\rangle_{h\in\BN}$ of elements of $\FM^n$ is $B$-indiscernible if every two
strictly increasing finite subsequences of $\cu I$ of the same length have the same type over $B$.
Indiscernible means $\emptyset$-indiscernible.

\begin{lemma}  \label{l-seq-set}  Suppose $T$ is stable.  Then every $B$-indiscernible sequence is a $B$-indiscernible set.
\end{lemma}

\begin{proof} The proof is the same as the proof of the corresponding result in first order logic (see, for example,
Proposition 7.1 in [Pi83]), but using the continuous notion of a stable formula in Definition \ref{d-stable-local}
instead of the first order notion of a stable formula.
\end{proof}

The next lemma is a continuous analogue of Definition III.1.5 and Lemma III.1.7 in [Sh78].

\begin{lemma}  \label{l-ave}
Suppose $T$ is stable, $n\in\BN$, $\cu I\subseteq\FM^n$  is an infinite indiscernible set, and $B\subseteq \FM$ is small.
There is a unique type $\Av(\cu I,B)\in S_k(B)$ such that for each formula $\psi(\vec x,B)$ with $|\vec x|=n$ and $\varepsilon>0$,
for all but finitely many $\vec a\in\cu I$ we have
\[
|\psi^{\FM}(\vec a,B)-\Av(\cu I,B)^\psi| <\varepsilon.
\]
\end{lemma}

\begin{proof}  Consider a formula $\psi(\vec x,B)$.  We have   $\psi^{\FM}(\vec x,B)=\psi^{\FM}(\vec x,\vec b)$ for some finite $\vec b\subseteq B$.
It suffices to show that there is a unique value $t(\psi)\in[0,1]$ such that for every $\varepsilon>0$, for all but finitely many $\vec a\in\cu I$ we have
$$|\psi^\FM(|\vec a,\vec b)-t(\psi)|<\varepsilon,$$
because then
$$\Av(\cu I,B)=\{\psi(\vec x,\vec b)\colon t(\psi)=0\}$$
 has the required property.
It is clear that  there is at most one such value $t(\psi)$.

Let is call a set $X\subseteq [0,1]$ \emph{broad} if $X$ is a closed interval, and $\psi^{\FM}(\vec a,\vec b)\in X$ for all but finitely many $\vec a\in\cu I$.
Note that any intersection of finitely many broad sets is broad.

\textbf{Claim \ref{l-ave}.1:}  Suppose $0\le s<r\le 1$.
Then at least one of the intervals $[0,r], [s,1]$ is broad.

Proof of Claim \ref{l-ave}.1.  Assume not.   Then there are infinitely many $\vec a\in\cu I$ such that $\psi^{\FM}(\vec a,\vec b) < s$, and infinitely many $\vec c\in\cu I$
such that $\psi^{\FM}(\vec c,\vec b)>r$.
Hence there is an infinite subsequence $\langle \vec c_h\rangle_{h\in\BN}$ of $\cu I$  such that for all even $h\in\BN$ we have
$\psi^{\FM}(\vec c_h,\vec b) < s$, and for all odd $j\in\BN$ we have $\psi^{\FM}(\vec c_j,\vec b)>r$.  Taking $\vec b_h=\vec b$ for all $h$, we see that
$\psi(\vec x,\vec y)$ is unstable in $T$.  But since $T$ is stable by hypothesis, $\psi(\vec x,\vec y)$ is stable in $T$
by Fact \ref{f-stable}.  This proves Claim \ref{l-ave}.1.

\textbf{Claim \ref{l-ave}.2:}  For each positive $k\in\BN$ there exists a broad set $X_k$ of length $\le 2/k$.

Proof of Claim \ref{l-ave}.2.  There is a least integer $m_0\le k$ such that $[0,m_0/k]$ is broad, and a greatest integer $m_1\le k$
such that $[m_1/k,1]$ is broad.  Then the set
$$X_k=[0,m_0/k]\cap[m_1/k,1]=[m_1/k,m_0/k]$$
is broad, but neither of the sets $[0,(m_0-1)/k]$, $[(m_1+1)/k,1]$ is broad.   By Claim \ref{l-ave}.1, we cannot have $m_0 - 1 > m_1 + 1$.  Hence $m_0\le m_1 + 2$, so
the interval $X_k$ has length $\le 2/k$.  This proves Claim \ref{l-ave}.2.

Since $[0,1]$ is compact,
$\bigcap_{n=1}^\infty X_k$ is non-empty, and contains exactly one point $t\in[0,1]$.  The result now follows with $t(\psi)=t$.
\end{proof}

\begin{cor}  \label{c-parallel}
Suppose $T$ is stable, $n\in\BN$, $\cu I\subseteq\FM^n$  is an infinite indiscernible set, and $B\subseteq \FM$ is small.
Then:
\begin{itemize}
\item[(i)] If $\cu I$ is $B$-indiscernible, then each $\vec a\in\cu I$ realizes $\Av(\cu I,B)$.
\item[(ii)]  For any infinite $\cu H\subseteq\cu I$, $\Av(\cu H,B)=\Av(\cu I,B)$.
\end{itemize}
\end{cor}

\begin{proof}  (i) is clear.

(ii)  By Lemma \ref{l-ave} we have $\Av(\cu I,B)\subseteq\Av(\cu H,B)$.  Since $\Av(\cu I,B)\in S_n(B)$ and $\Av(\cu H,B)\in S_n(B)$, $\Av(\cu H,B)=\Av(\cu J,B)$.
\end{proof}

We will only need the following lemma for subsets of  $\FM$, but a slightly more complicated argument will prove
the corresponding result for subsets of $\FM^n$.

\begin{lemma}  \label{l-morley-seq}  Suppose $T$ is stable, and $C$ is small and algebraically closed.
For each element $a\in\FM\setminus C$ there is an infinite
$C$-indiscernible set  $\cu I\subseteq\FM$  such that $a\in\cu I$, and
$X\ind_C Y$ for any  disjoint $X,Y\subseteq \cu I$.
\end{lemma}

\begin{proof}  The proof is the same as the corresponding argument in first order logic.
By Existence and Symmetry, for any $X$ and $Y$ we have $X\ind_C \emptyset$ and $\emptyset\ind_C Y$.
By Existence, one can inductively build a sequence $\cu I=\langle a_h\rangle_{h\in\BN}$ in $\FM$  such that $a_0=a$, and for each $n\in\BN$ we have
$$ a_n\equiv_C  a \mbox{ and } a_n \ind_C\{ a_h\colon h<n\}.$$
By Anti-reflexivity and Finite Character, $a_n\ne a_h$ whenever $h<n$, so $\cu I$ is infinite.
Using Symmetry and Transitivity for  $\ind$, one can show by induction on $n$ that $X\ind_C Y$ for each
pair $X, Y$ of disjoint subsets of $\{ a_0,\ldots, a_n\}$.
By Symmetry and Finite Character, it follows that $X\ind_C Y$ for any  disjoint $X,Y\subseteq \cu I$.

By Lemma \ref{l-seq-set}, to show that $\cu I$ is a $C$-indiscernible set, it suffices to show that $\cu I$ is a $C$-indiscernible sequence.
This means that for all $n\in\BN$, $X\equiv_C Y$ for all strictly increasing subsequences
$$X=\{x_1,\ldots,x_n\},  Y=\{y_1,\ldots,y_n\}$$
of $\cu I$ of length $n$.  We argue by induction on $n$.  The result holds trivially for $n=1$.
Suppose  $(X,u)$, $(Y,v)$ are strictly increasing subsequences of $\cu I$ of length $n+1$, and $X\equiv_C Y$.  Without loss of generality
we may assume that $u$ is equal to or  before $v$ in $\cu I$.  Then $u\equiv_C v$,
$X\ind_C v$, and $Y\ind_C v$, so by Stationarity we have $X\equiv_{Cv} Y$ and hence $(X,v)\equiv_C (Y,v)$.
By a similar argument but using Symmetry, we have  $u\equiv_{CX} v$, and hence
$$(X,u)\equiv_C (X,v)\equiv_C  (Y,v).$$
This completes the induction.
\end{proof}

The sequence $\cu I$ in the above proof is called a \emph{Morley sequence}  in type $\tp(a_0/C)$.
A proof of the following fact for first order logic using Morley sequences can be found, for example, in [Ad09].  The same proof works for continuous logic (see also [EG12]).

\begin{fact}   Every stable independence relation $\ind$ has the following additional properties:
\begin{itemize}
\item Base Monotonicity: If $C\subseteq D\subseteq B$ and $A\ind_C B$, then $A\ind_D B$.
\item Normality: If $A\ind_C B$ then $AC\ind_C B$.
\item Anti-reflexivity: If $a\ind_C a$ then $a$ belongs to the algebraic closure of $C$.
\end{itemize}
\end{fact}

\begin{lemma}  \label{l-remove}  Suppose $T$ is stable and $\cu I$ is an indiscernible set.  Then for any set $B\subseteq\FM$
there is a set $X\subseteq\cu I$ such that $|X|\le |B|+\aleph_1$, and $\cu I\setminus X$ is $(B\cup\bigcup X)$-indiscernible.
\end{lemma}

\begin{proof}  The argument is exactly the same as the proof of Corollary III.3.5 of [Sh78], but applied to continuous rather than first order logic.
\end{proof}

\begin{lemma}  \label{l-maximal-indiscernible}  Suppose $T$ is stable, $\cu M\prec\FM$ is small and $\aleph_1$-saturated,
and  every indiscernible set   $\cu I\subseteq \cu M$ of cardinality $\aleph_0$ can be extended to an
indiscernible set  $\cu H\subseteq\cu M$
of cardinality $\kappa$.   Then $\cu M$ is $\kappa$-saturated.
\end{lemma}

\begin{proof} The proof is similar to the proof of Lemma III.3.10 in [Sh78].
Suppose $\cu M$ is not $\kappa$-saturated.  Then $\kappa > \aleph_1$, and there is a  set $B\subseteq M$ of cardinality $|B|<\kappa$ and an element $a\in\FM$
such that no element of $M$ realizes $\tp(a/B)$.  By Strong Local Character and Base Monotonicity, there is a countable algebraically closed set $C\subseteq M$
such that $a\ind_C M$.

By Lemma \ref{l-morley-seq}, there is a Morley sequence $\cu I=\langle a_h\rangle_{h\in\BN}$ in $\tp(a/M)$, with $a_0=a$.
By Corollary \ref{c-parallel} (i), each $a_h\in\cu I$ realizes $\Av(\cu I,M)$.
Let $\cu I_e=\{a_h\colon h \mbox{ is even}\}$ and $\cu I_o=\{a_h\colon h \mbox{ is odd}\}$.
Since $\cu M$ is $\aleph_1$-saturated and $C$ is countable, there is a sequence $\cu I'$ of elements of $M$ such that
$(\cu I_e,\cu I_o)\equiv_C (\cu I_e,\cu I')$.  Then $\cu I_e\cup \cu I'$ is $C$-indiscernible.  By Corollary \ref{c-parallel} (ii),
$$\Av(\cu I,M)=\Av(\cu I_e,M)=\Av(\cu I',M).$$
By hypothesis, $\cu I'$ can be extended to an indiscernible set $\cu H\subseteq M$ of cardinality $\kappa$.
By Lemma \ref{l-remove}, there is a set $X\subseteq\cu H$ such that $|X|\le|B C|+\aleph_1<\kappa$ and $\cu H'=\cu H\setminus X$ is $BC$-indiscernible.
Then $\cu H'$ is infinite.  By Corollary \ref{c-parallel} (ii),
$$\Av(\cu I',M)=\Av(\cu H,M)=\Av(\cu H',M),$$
so
$$\Av(\cu I,BC)=\Av(\cu H',BC).$$
Pick an element $a'\in\cu H'$.  Then $a'\in M$.
By Corollary \ref{c-parallel} (i), $a$ realizes $\Av(\cu I,BC)$ and $a'$ realizes $\Av(\cu H',BC)$.
But then $a\equiv_{BC} a'$, which contradicts the fact that no element of $M$ realizes $\tp(a/B)$.
So $\cu M$ is $\kappa$-saturated after all.
\end{proof}

The following definition is from Shelah [Sh78].

\begin{df} Let $\cu D$ be a regular ultrafilter.  The \emph{lower cofinality} $\lcf(\cu D)$ is the  least cardinal $\kappa$ such that there is a
co-initial set $C$ of $\kappa$ infinite elements of the ultrapower $(\BN,\le)_\cu D$.  That is, $C$ is a set of  infinite elements of $(\BN,\le)_\cu D$,
$|C|=\kappa$, and for each infinite element $e\in(\BN,\le)_\cu D$,
we have $(\cu N,\le)_{\cu D}\models c\le e$ for some $c\in C$.
\end{df}

\begin{df}  Let $\cu D$ be a regular ultrafilter.  The \emph{lower cardinality} $\lca(\cu D)$ is the  least cardinality of an infinite initial segment of
the ultrapower $(\BN,\le)_\cu D$.  That is, the least infinite cardinal $\lambda$ such that
for some $a\in (\BN,\le)_\cu D$, $\lambda=|\{b\colon (\BN,\le)_\cu D\models b\le a\}|$.
\end{df}

Note that for every regular ultrafilter $\cu D$ over $I$, we have
$$\aleph_0<\lcf(\cu D)\le\lca(\cu D)\le |2^I|.$$

\begin{fact}  \label{f-exists-D} (Theorem VI.3.12 in [Sh78]).  If $\aleph_0<\kappa\le\mu=\mu^{\aleph_0}\le |2^I|$, then there is a regular ultrafilter $\cu D$ over $I$
such that $\lcf(\cu D)=\kappa$ and $\lca(\cu D)=\mu$.
\end{fact}

\begin{df}  Let  $\varepsilon\in[0,1]$, and $\Delta$ be a  finite set of formulas  with free variables included in $\vec y$ that contains at least the formula $d(y_0,y_1)$.
 A sequence $\cu I\subseteq \cu M$ is $(\Delta,\varepsilon)$-\emph{indiscernible} in $\cu M$ if
for all proper $|\vec y|$-tuples $\vec b,\vec c$
of elements of $\cu I$ and every $\varphi(\vec y)\in\Delta$, we have
$$|\varphi(\vec b)^{\cu M} - \varphi(\vec c)^{\cu M}|\le\varepsilon.$$
$\cu I$ is $r$-\emph{separated} if $d^{\cu M}(u,v)\ge r$ for some $u,v\in\cu I$.
\end{df}

We make some easy observations about $(\Delta,\varepsilon)$-indiscernibility.  Let $1\ge r>\varepsilon$.
If $\cu I$ is $r$-separated, then any superset of $\cu I$ in $\cu M$ is $r$-separated.  Since $d(y_0,y_1)$ is in $\Delta$, if  $\cu I$ is $(\Delta,\varepsilon)$-indiscernible
in $\cu M$ and $r$-separated,
 then two elements $u,v\in\cu I$ are distinct if and only if $d^{\cu M}(u,v)\ge r-\varepsilon$.

For each $k\in\BN$, there is a single formula $\varphi(\vec z)$ with $|\vec z|=k$ such that for each $k$-tuple $\cu H$ in $\cu M$,
$\varphi^\FM(\cu H)$ holds
if and only if $\cu H$ is $(\Delta,\varepsilon)$-indiscernible in $T$ and $r$-separated.
$\cu I$ is indiscernible in $T$ if and only if it is $(\Delta,1/n)$-indiscernible in $T$ for all non-empty finite $\Delta$ and all $0<n\in\BN$.
The property of $(\Delta,\varepsilon)$-indiscernibility gets stronger as $\Delta$ gets larger and $\varepsilon$ gets smaller.

The proof of the next lemma  is similar to  the argument for Theorem VI.5.1 (2) in [Sh78], but is simpler because the argument in [Sh78]
also dealt with three other more difficult cases at the same time.

\begin{lemma}  \label{l-stable-saturated}  If $T$ is stable, $\cu D$ is a regular ultrafilter, and $\cu M\models T$, then
 $\cu M_\cu D$ is $\lca(\cu D)$-saturated.
\end{lemma}

\begin{proof}
  By Remark \ref{r-regular-aleph-1}, $\cu M_\cu D$ is $\aleph_1$-saturated.
By Lemma \ref{l-maximal-indiscernible}, it suffices to prove that  every countably infinite indiscernible set $\cu I=\{a_h\colon h\in\BN\}$ in $\cu M_{\cu D}$
 can be extended to an indiscernible set $\cu H$ in $\cu M_\cu D$ of cardinality $\ge\lca(\cu D)$.

$\cu I$ is $2r$-separated for some $1\ge r>0$.
Let $\cu M'=(\cu M, S,\in)$ be the two-sorted real-valued structure where one sort is $M$, the other sort  is the set $S$ of finite subsets of $M$,
and $\in$ is the $\in$ relation on $M\times S$. ($\cu M'$ is not a metric structure.) Let $V$ and $V'$ be the vocabularies of $\cu M$ and $\cu M'$.
Let $\bo z$ be a variable of sort $S$.
For every  finite set $\Delta(\vec y)$ of strict $V$-formulas containing $d(y_0,y_1)$ and every $1/r<n\in\BN$,
there is a $V'$-formula $\varphi_{\Delta,n,r}(\bo z)$ saying that the set
$\bo z$ is $(\Delta,1/n)$-indiscernible and $r$-separated.
By Remark \ref{r-regular-aleph-1}, the ultrapower $\cu M'_{\cu D}$ is $\aleph_1$-saturated.
The  set of $V'$-formulas
$$\Gamma(\bo z)=\{a_h\in \bo z\colon h\in\BN\}\cup\{\varphi_{\Delta,n,r}(\bo z)\colon \Delta \mbox{ finite and strict}, 0<n\in\BN\}$$
is finitely satisfiable in $\cu M'_{\cu D}$, because each finite subset $\Gamma_0\subseteq\Gamma$ is satisfied by
the element $b_{0\cu D}$ of $S_{\cu D}$ where for each $t\in I$, $b_0[t]$ is the set of $a_h[t]$ such that $a_h$ occurs in $\Gamma_0$.
$\Gamma(\bo z)$ is countable because there are only countably many strict $V$-formulas.
Therefore $\Gamma(\bo z)$ is satisfied in $\cu M'_{\cu D}$ by some element $b_{\cu D}$ of sort $S$ in $\cu M'_{\cu D}$.
It follows that the set
$$\cu H=\{a\in M_{\cu D}\colon \cu M'_{\cu D}\models a\in b\}$$
contains $\cu I$ and is indiscernible in $\cu M_{\cu D}$.  Since $\cu H$ is infinite, we have
$$|\cu H|=\left|\prod_{\cu D} b[t]\right|\ge \lca(\cu D),$$
as required.
\end{proof}

\begin{lemma}  \label{l-lcf}
Suppose $\cu D$ is a regular ultrafilter, and  the ultrapower $\cu M_\cu D$ is  $\lcf(\cu D)^+$-saturated.  Then $\Th(\cu M)$ is stable.
\end{lemma}

\begin{proof}  Suppose $\varphi(\vec x,\vec y)$  is an unstable formula in $\Th(\cu M)$.  By Theorem \ref{t-equiv-triangle}, we may assume that $\cu M\prec\FM$.
There are  dyadic rationals $r>s$ in $[0,1]$,
and a sequence $\langle (\vec a_k,\vec b_k)\rangle_{k\in\BN}$ of tuples in $\cu M$ such that whenever $h< j<k$ in $\BN$ we have
$\varphi^{\cu M}(\vec a_j,\vec b_h)\ge r$ and $\varphi^{\cu M}(\vec a_j,\vec b_k)\le s$.  For simplicity we give the proof in the case that
$\vec x,\vec y$ are $1$-tuples, and that $\varphi^{\cu M}(a_j,b_h)\ge r$ and $\varphi^{\cu M}(a_j,b_k)\le s$ whenever $h\le j<k$.  (The proof in general is
essentially the same).  We may take the $b_k$ to be distinct.
Consider the countable first order structure $(K,\le^K)$ where $K=\{b_m\colon m\in\BN\}$ and $\le^K=\{(b_m,b_n)\colon m\le n\}$.
Let $f$ be the unique isomorphism from $(\BN,\le)$ onto $(K,\le^K)$, and $f_\cu D$ be the corresponding isomorphism from
$(\BN,\le)_\cu D$ onto $(K,\le^K)_\cu D$.  Note that $K\subseteq M$, so $K_{\cu D}\subseteq M_{\cu D}$.
Let $\kappa=\lcf(\cu D)$.  There is a co-initial set $C$ of infinite elements of $(\BN,\le)_\cu D$ such that
$|C|=\kappa$.  Let $\Sigma(x)$ be the set of formulas
$$\Sigma(x)=\{r\dotle\varphi(x,f_\cu D(n))\colon n\in\BN\}\cup\{\varphi(x,f_\cu D(c))\dotle s\colon c\in C\}.$$
(By Remark \ref{r-dotminus}, we are using $\dotle$ as an alternate notation for $\dotminus$).
Then $\Sigma(x)$ is finitely satisfiable in $\cu M_\cu D$ and has cardinality $\kappa$.

We claim that $\Sigma(x)$ is not satisfiable in $\cu M_\cu D$.  To see this, suppose
$a$ satisfies $\Sigma(x)$ in $\cu M_\cu D$.  Then for some $X\in\cu D$, for every $t\in X$ there exists $e[t]\in\BN$ such that
$\varphi^\cu M(a[t],b_{e[t]})\ge r$ and $\varphi^\cu M(a[t],b_{e[t]+1})\le s$.
Then $\varphi^{\cu M_\cu D}(a,f_\cu D(e))\ge r$ and $\varphi^{\cu M_\cu D}(a,f_\cu D(e+1))\le s$.  Moreover,
$e$ is an infinite element of $(\BN,\le)_\cu D$, so $c\le_\cu D e$ for some $c\in C$.  Then $\varphi^{\cu M_\cu D}(a,f_\cu D(c))\ge r$, which contradicts
the assumption that $a$ satisfies $\Sigma(x)$ in $\cu M_\cu D$.  This shows that $\Sigma(x)$ is not satisfiable in $\cu M_\cu D$, so
$\cu M_\cu D$ is not $\kappa^+$-saturated.
\end{proof}

\begin{df}  $LCA$ is the class of regular ultrafilters $\cu D$ over sets $I$ such that $|I|<\lca(\cu D)$.

$LCF$ is the class of regular ultrafilters  $\cu D$
over sets $I$ such that $|I|<\lcf(\cu D)$.
\end{df}

By Fact \ref{f-exists-D}, $LCF$ is a proper subclass of $LCA$.  In fact, whenever $\aleph_1<2^{|I|}$,
there is a regular ultrafilter $\cu D$ over $I$ such that $\cu D\in LCA\setminus LCF$.

\begin{thm} \label{t-stable-vs-unstable}

\noindent\begin{itemize}
\item[(i)]  $T$ is stable if and only if $\Sat(T)\supseteq LCA$.
\item[(ii)] $T$ is unstable if and only if $\Sat(T)\subseteq LCF$.
\item[(iii)] If  $T$ is stable and $U$ is unstable, then $T\lk U$.
\end{itemize}
\end{thm}

\begin{proof}  Let $\cu M$ be a model of $T$.

(i) Forward:  Suppose $T$ is stable.  By Lemma \ref{l-stable-saturated},
$\cu M_{\cu D}$ is $\lca(\cu D)$-saturated.
If $\cu D\in LCA$, then $|I|<\lca(\cu D)$, so $\cu M_{\cu D}$ is $|I|^+$-saturated and $\cu D\in\Sat(T)$.  Therefore $\Sat(T)\supseteq LCA$.

(ii)  Suppose $T$ is unstable.  By Lemma \ref{l-lcf}, $\cu M_\cu D$ is not $\lcf(\cu D)^+$-saturated.
If $\cu D\in\Sat(T)$, then $\cu M_{\cu D}$ is $|I|^+$-saturated, so $|I|<\lcf(\cu D)$ and $\cu D\in LCF$.  Thus $\Sat(T)\subseteq LCF$.
Now suppose $\Sat(T)\subseteq LCF$.  Then $\Sat(T)\supseteq LCA$ fails, so $T$ is unstable by (i) forward,

(i) Reverse:  Suppose $T$ is unstable. Take $\cu D\in LCA\setminus LCF$.  By (ii),  $\cu D \notin\Sat(T)$, so $\Sat(T)\supseteq LCA$ fails.

(iii) follows by (i), (ii), and Fact \ref{f-exists-D}.
\end{proof}

\begin{cor}  \label{c-stable-vs-unstable}
There are regular ultrafilters $\cu D$ that saturate all stable metric theories but no unstable metric theories.
The class of all such $\cu D$ is $LCA\setminus LCF$.
\end{cor}

\begin{proof}
Let $\stb_\BM$ be the set of $\lek$-equivalence classes of stable metric theories in $\BM$, and
$\stb_\BF$ be the set of $\lek$-equivalence classes of stable first order theories in $\BF$.
By Fact \ref{f-exists-D},
$$\{\cu D\colon \Sat_\BF(\cu D)=\stb_\BF\}=LCA\setminus LCF\ne\emptyset.$$
By Theorem \ref{t-stable-vs-unstable}, we have
$$\{\cu D\colon \Sat_\BF(\cu D)=\stb_\BF\}=\{\cu D\colon \Sat_\BM(\cu D)=\stb_\BM\}.$$
\end{proof}

\section{$\lek$-minimal theories}

Theorem \ref{t-minimal=fsp} below gives a characterization of $\lek$-minimal metric theories.  Theorem
\ref{t-stable-notmin} shows that there are exactly two $\lek$-classes of stable metric theories, and that
they are the lowest two $\lek$-equivalence classes.  We are indebted to James Hanson for
pointing out an error in the definition of NFCP and the proof of
Theorem \ref{t-minimal=fsp} in an earlier version of this section.  To correct that error we modified the definition
of NFCP from the earlier version.

For first order logic, in
[Sh78], Section VI, Shelah identified  the first two classes in the partial ordering $(\BF,\lek)$.

\begin{df}  \label{d-snm} We denote the set  of   first order theories that are stable but not $\lek$-minimal by $\snm_\BF$.
We denote the set  of  metric theories that are stable but not $\lek$-minimal by $\snm_\BM$.
\end{df}

\begin{fact}  \label{f-shelah} (Theorem VI.5.9 in [Sh78]),

(i) $\snm_\BF\in \BF$.

(ii)  For each $H\in\BF$, either $H=\min_\BF$, $H=\snm_\BF$,
or $\snm_\BF\lk H$.
\end{fact}

In Theorem \ref{t-stable-notmin} and Corollary \ref{c-first-two} below, we will prove the analogous results for metric theories.
In this section, $T$ will be a complete metric theory and $\cu M$ will denote an $\aleph_1$-saturated model of $T$.

\begin{df}  Let $\varepsilon\in[0,1]$ and  $\Sigma(\vec y)$ be a set of formulas with parameters in $\cu M$
(where $\vec y$ is a possibly infinite sequence of variables).  We say that $\Sigma(\vec y)$ is $\varepsilon$-\emph{satisfiable} in $\cu M$
if there is a $|\vec y|$-sequence of elements $\vec a$ of $\cu M$
such that $\psi^{\cu M}(\vec a)\le\varepsilon$ for all $\psi\in\Sigma(\vec y)$.
\end{df}

\begin{cor}  \label{c-varepsilon-sat}  Let $\Sigma(\vec y)$ be a set of formulas with parameters in $\cu M$.

(i)  $\Sigma(\vec y)$ is $\varepsilon$-satisfiable in $\cu M$ if and only if the set of formulas
$$\{\psi(\vec y)\dotminus \varepsilon\colon\psi\in\Sigma\}$$
is satisfiable in $\cu M$.

(ii)  If $\Sigma(\vec y)$ is finite, $\Sigma(\vec y)$ is $\varepsilon$-satisfiable in $\cu M$ if and only if the sentence
$(\inf_{\vec y})\max(\Sigma(\vec y))\dotminus\varepsilon$ holds in $\cu M$.

(iii)  If $\Sigma(\vec y)$ is countable and every finite subset of $\Sigma(\vec y)$ is $\varepsilon$-satisfiable in $\cu M$,
then $\Sigma(\vec y)$ is $\varepsilon$-satisfiable in $\cu M$.
\end{cor}

\begin{proof}
(i) is a restatement of the definition, and
(ii) follows easily from the fact that $\cu M$ is $\aleph_1$-saturated.

(iii): By (i), the countable set of formulas
$$\{\psi(\vec y)\dotminus\varepsilon\colon\psi(\vec y)\in\Sigma(\vec y)\}$$
is finitely satisfiable in $\cu M$.  Since $\cu M$ is $\aleph_1$-saturated, that set of formulas is satisfiable in $\cu M$.
By (i), $\Sigma(\vec y)$ is $\varepsilon$-satisfiable in $\cu M$.
\end{proof}

\begin{df}
Let $\vec x$ be an $n$-tuple of distinct variables, $B\subseteq \cu M$, and $\Delta(\vec x)$ be a finite set of formulas with free variables among $\vec x$.
By an $n$-\emph{tuple over} $(\vec y,B)$ we mean an $n$-tuple $(u_0,\ldots,u_{n-1})$ where each $u_i$ is in $\vec y\cup B$,
and no variable or parameter appears more than once in $\vec u$.
By a \emph{$\Delta$-formula over $(\vec y,B)$}
we mean a formula $\psi(\vec u)$ obtained by a formula $\psi(\vec x)\in\Delta$ by replacing $\vec x$ by an $n$-tuple $\vec u$  over $(\vec y,B)$.

In the case that $B=M$, we sometimes say ``over $(\vec y,\cu M)$'' instead of ``over $(\vec y,B)$''.
\end{df}

\begin{df}   \label{d-fsp}  \label{d-NFCP}
$T$ has the  \emph{non-finite cover property} (NFCP) in $\cu M$ (a better name would be ``finite satisfaction property'') if
$T$ is stable, and for each  non-empty finite set $\Delta(\vec x)$ of formulas and $\varepsilon \in(0,1]$,
there is an $m(\Delta,\varepsilon)\in\BN$ such that for every
finite set $\Sigma(\vec y)$ of $\Delta$-formulas over $(\vec y,\cu M)$,
 if every subset of $\Sigma(\vec y)$ of cardinality $\le m(\Delta,\varepsilon)$
is satisfiable in $\cu M$, then $\Sigma(\vec y)$ is $\varepsilon$-satisfiable in $\cu M$.

$T$ has the NFCP if $T$ has the NFCP in every $\aleph_1$-saturated model of $T$.
\end{df}

We will see in Corollary \ref{c-FO-minimal} below that a first order theory has the NFCP if and only if it does not have the finite cover property as defined in [Ke67].

\begin{lemma}  \label{l-NFCP-every}
If $T$ has the NFCP in some $\aleph_1$-saturated model of $T$, then $T$ has the NFCP.
\end{lemma}

\begin{proof}  Suppose $T$ has the NFCP in  $\cu M$, and $\cu N$ is another $\aleph_1$-saturated model of $T$.
For every countable sequence $\vec b\subseteq\cu M$ there is a countable sequence $\vec c\subseteq N$ such that $(\cu M,\vec b)\equiv(\cu N,\vec c)$.
By Corollary \ref{c-varepsilon-sat}, a countable set of formulas $\Sigma(\vec y,\vec b)$ is $\varepsilon$-satisfiable in $\cu M$
if and only if $\Sigma(\vec y,\vec c)$ is $\varepsilon$-satisfiable in $\cu N$.
It follows that $T$ has the NFCP in $\cu N$ as well as in $\cu M$.
\end{proof}

\begin{cor}  \label{r-fcp-big}  Suppose $T$ has the NFCP and $\Delta(\vec x)$, $\varepsilon$, and $m(\Delta,\varepsilon)$ are
as in Definition  \ref{d-fsp}.  Then for every countable  set $\Sigma(\vec y)$ of $\Delta$-formulas over $(\vec y,\cu M)$,
 if every subset of $\Sigma(\vec y)$ of cardinality $\le m(\Delta,\varepsilon)$
is satisfiable in $\cu M$, then $\Sigma(\vec y)$ is $\varepsilon$-satisfiable in $\cu M$.
\end{cor}

\begin{proof}  By Corollary \ref{c-varepsilon-sat}.
\end{proof}

\begin{lemma}  \label{l-extend-indiscernible}
Suppose $T$ has the NFCP. Then for each $ r/2>\varepsilon'>\varepsilon\ge 0$ and
 finite set of formulas $\Delta(\vec x)$ containing $d(x_0,x_1)$,
there exists $k\in\BN$
such that  every $r$-separated finite $(\Delta,\varepsilon)$-indiscernible set in $\cu M$ of cardinality $\ge k$ can be extended to an
$(\Delta,\varepsilon')$-indiscernible set in $\cu M$ of cardinality $\aleph_0$.
\end{lemma}

\begin{proof}
Let $n=|\vec x|$ and $\vec x'$ be an $n$-tuple of new distinct variables, and let $\Delta'$ be the finite set of formulas
$$\Delta'=\Delta\cup\{|\varphi(\vec x)-\varphi(\vec x')|\dotminus\varepsilon\colon \varphi(\vec x)\in\Delta\}.$$

Suppose that $\cu I$ is $(\Delta,\varepsilon)$-indiscernible in $\cu M$, $r$-separated, and at most countable.  Let
$\vec y$ be a countable sequence of variables and let
$$\Sigma(\vec y,\cu I)=\{|\varphi(\vec u)-\varphi(\vec v)|\dotminus\varepsilon\colon  \vec u,\vec v
\mbox{ are $n$-tuples over } (\vec y,\cu I) \mbox{ and } \varphi\in\Delta\}.$$
Then $\Sigma(\vec y,\cu I)$ is a countable set of $\Delta'$-formulas over $(\vec y,\cu I)$.
\medskip

\textbf{Claim \ref{l-extend-indiscernible}.1.} \emph{Assume that $\vec a$ is a  countable sequence in $\cu M$ that $(\varepsilon'-\varepsilon)$-satisfies $\Sigma(\vec y,\cu I)$.
Then $\vec a$ is disjoint from $\cu I$ and
$\cu I\cup\vec a$ is $(\Delta,\varepsilon')$-indiscernible in $\cu M$.}

\begin{proof} [Proof of Claim \ref{l-extend-indiscernible}.1]
Work in $\cu M$.  Since $\cu I$ is $r$-separated, for each distinct $b,c$ in $\cu I$ we have $d(b,c)\ge r-\varepsilon$.
Then for each $a_i\in\vec a$ and $c\in\cu I$  we have $|d(a_i,c)-d(b,c)|\le\varepsilon'$, so $d(a_i,c)\ge r-2\varepsilon'>0$ and hence $a_i\notin\cu I$.
For each pair of $n$-tuples $\vec u,\vec v$ over $(\vec y,\cu I)$ and $\varphi(\vec x)\in\Delta$,
 the formula $|\varphi(\vec u)-\varphi(\vec v)|\dotminus\varepsilon$
belongs to $\Sigma(\vec y,\cu I)$.  Since $\vec a$ $(\varepsilon'-\varepsilon)$-satisfies $\Sigma(\vec y,\cu I)$
and $\cu I$ is $(\Delta,\varepsilon)$-indiscernible, $\vec a$ satisfies
$|\varphi(\vec u)-\varphi(\vec v)|\le\varepsilon'.$  Therefore $\cu I\cup\vec a$ is $(\Delta,\varepsilon')$-indiscernible in $\cu M$.
This proves Claim \ref{l-extend-indiscernible}.1.
\end{proof}

Now let $m:=m(\Delta',\varepsilon'-\varepsilon)$ and $k:= 4mn$.
\medskip

\textbf{Claim \ref{l-extend-indiscernible}.2.} \emph{Assume $\cu I$ has cardinality $\ge k$.
Then every subset of $\Sigma(\vec y,\cu I)$ of cardinality $\le m$ is satisfiable in $\cu M$.}

\begin{proof} [Proof of Claim\ref{l-extend-indiscernible}.2]
 Let $\Sigma_0(\vec z)$ be a subset  of $\Sigma(\vec y,\cu I)$ of cardinality $\le m$
where $\vec z$ is the finite subsequence of $\vec y$ whose terms actually occur in $\Sigma_0(\vec z)$.  $\Sigma_0(\vec z)$ is contained in
$\Sigma(\vec y,\cu H)$ for some $\cu H\subseteq\cu I$
of cardinality at most $2mn$.  Also, the length of $\vec z$ is at most $2mn$.
Since $\cu I$ has cardinality $\ge k=4mn$, there is a sequence $\vec c$ of distinct elements of $\cu I\setminus\cu H$ of length $|\vec z|$.
Since $\cu I$ is $(\Delta,\varepsilon)$-indiscernible in $\cu M$, $\Sigma_0(\vec z)$ is satisfied by $\vec c$ in $\cu M$.
 This proves Claim \ref{l-extend-indiscernible}.2.
\end{proof}

By Corollary \ref{r-fcp-big} and Claim \ref{l-extend-indiscernible}.2, $\Sigma(\vec y,\cu I)$
is $(\varepsilon'-\varepsilon)$-satisfiable by some infinite sequence $\vec a$ in $\cu M$.
Then by Claim \ref{l-extend-indiscernible}.1, $\cu I\cup\vec a$ is $(\Delta,\varepsilon')$-indiscernible in $\cu M$ and has cardinality $\aleph_0$.
\end{proof}

The following theorem gives a characterization of $\lek$-minimal theories.

\begin{thm}  \label{t-minimal=fsp}
Let $T=\Th(\cu M)$. The following are equivalent:
\begin{itemize}
\item[(i)]  $T$ is $\lek$-minimal.
\item[(ii)]  $\Sat(T)\supsetneqq LCA$.
\item[(iii)]  $T$ has the NFCP.
\item[(iv)]  $T$ is stable and satisfies the conclusion of Lemma \ref{l-extend-indiscernible}.
\end{itemize}
\end{thm}

\begin{proof}
(i) $\Rightarrow$ (ii):  Assume (i).  By Lemma \ref{l-min-max-G}, $\Sat(T)$ is the class of all regular ultrafilters, so $\Sat(T)\supseteq LCA$.
By Fact \ref{f-exists-D}, if $|I|\ge 2^{\aleph_0}$ then there is a regular ultrafilter $\cu D$ over $I$ such that $\lca(\cu D)=2^{\aleph_0}\le|I|$,
so $\cu D\in\Sat(T)\setminus LCA$ and (ii) holds.

(ii) $\Rightarrow$ (iii): Assume (ii) holds but (iii) fails.  By (ii) and  Theorem \ref{t-stable-vs-unstable} (i), $T$ is stable.
By (ii), there is an ultrafilter  $\cu D$  over $I$ such that $\cu D\in\Sat(T)\setminus LCA$. Since $\cu D\notin LCA$, $\lca(\cu D)\le|I|$.
Since (iii) fails, there is non-empty finite set of formulas $\Delta(\vec x)$, and an $\varepsilon>0$
such that  for each $m\in \BN$ there is a finite tuple $\vec b_m$ of parameters in $\cu M$ and
a finite set  $\Sigma_m(\vec y,\vec b_m)$ of $\Delta$-formulas over $(\vec y,\vec b_m)$ with the following property:
\medskip

\noindent\emph{ Each subset of $\Sigma_m$ of cardinality $\le m$ is satisfiable in $\cu M$, but $\Sigma_m$ is not $\varepsilon$-satisfiable in $\cu M$.}
\medskip

Let $f(m)$ be the smallest cardinality of a set $\Sigma_m$ with the above property.  Then $m<f(m)\in\BN$, and in $\cu M$,
every subset of $\Sigma_m$ of cardinality $\le m$ is satisfiable,
and every proper subset of $\Sigma_m$ is $\varepsilon$-satisfiable, but
$\Sigma_m$ is not $\varepsilon$-satisfiable.  Whenever $f(0)<\ell\in\BN$, let $g(\ell)$ be the greatest $m\in\BN$ such that
$f(m)\le \ell$.  For each $\ell\le f(0)$, let $g(\ell)=0$.
Then $g(\ell)<\max(\ell,f(0))$ for all $\ell\in\BN$, and $\lim_{\ell\to\infty} g(\ell)=\infty$.

By the definition of $\lca(\cu D)$, there is an $\eta\in (\BN,\le)_\cu D$  such that
$$|\{\beta\colon(\BN,\le)_\cu D\models \beta\le\eta\}|=\lca(\cu D).$$
Then in $(\BN,\le,g)_\cu D$, $g(\eta)$ is infinite and $g(\eta)\le\eta$, so
$$|\{\beta\colon(\BN,\le,g)_\cu D\models \beta\le g(\eta)\}|=\lca(\cu D).$$
For each $t\in I$, let $\Sigma[t](\vec y)$ be the  set of formulas $\Sigma_{g(\eta[t])}(\vec y)$,
which is finite because $\eta[t]\in\BN$. Since $\Delta$ is finite, if $\psi[t](\vec x)\in\Delta$ for each $t\in I$ and
$$\{t\in I\colon\psi[t](\vec y,\vec b[t])\in \Sigma[t](\vec y)\}\in\cu D,$$
then there is a unique $\Delta$-formula $\theta(\vec y,\vec b_{\cu D})$ over $(\vec y,\cu M_{\cu D})$ such that
$$\{t\in I\colon\psi[t](\vec y,\vec b[t])=\theta(\vec y,\vec b[t]\}\in\cu D.$$

Let $\Sigma_{\cu D}(\vec y)$ be the set of all $\Delta$-formulas $\theta(\vec y,\vec b_{\cu D})$ over $(\vec y,\cu M_{\cu D})$ such that
$$ \{t\colon \theta(\vec y,\vec b[t])\in \Sigma[t])\}\in\cu D.$$
Then $|\Sigma_{\cu D}(\vec y)|=\lca(\cu D)$, and in $\cu M_{\cu D}$, every
finite subset of $\Sigma_{\cu D}(\vec y)$ is $\varepsilon$-satisfiable, but $\Sigma_{\cu D}(\vec y)$ is not $\varepsilon$-satisfiable.
By Remark \ref{r-regular-aleph-1}, the ultrapower $\cu M_{\cu D}$ is an $\aleph_1$-saturated model of $T$.
Let
$$\Gamma_{\cu D}(\vec y)=\{\theta(\vec y,\vec b_{\cu D})\dotminus\varepsilon\colon \theta(\vec y,\vec b_{\cu D})\in\Sigma_{\cu D}(\vec y)\}.$$
Then  $\Gamma_{\cu D}(\vec y)$ is a set of formulas with parameters in $\cu M_{\cu D}$ of cardinality $\lca(\cu D)$
that is finitely satisfiable but not satisfiable in $\cu M_{\cu D}$, so
$\cu M_{\cu D}$ is not $\lca(\cu D)^+$-saturated. Since $\lca(\cu D)\le{I}$,
$\cu M_{\cu D}$ is not ${I}^+$-saturated, contradicting the hypothesis that $\cu D\in\Sat(T)$. Therefore (ii) implies (iii).

By Lemma \ref{l-extend-indiscernible}, (iii) implies (iv).

(iv) $\Rightarrow$ (i):  Assume (iv).  
Our argument will be similar to the proof of Theorem VI.5.1 (1) of [Sh78] and of Lemma \ref{l-stable-saturated} above.
 Let $\cu D$ be a regular ultrafilter over $I$, and let $\cu M$ be an $\aleph_1$-saturated model of $T$.
We prove that $\cu M_\cu D$ is $2^{|I|}$-saturated, and hence $|I|^+$-saturated and $\cu D\in\Sat(T)$.  Since $\cu D$ is arbitrary, this will prove (i).

By Remark \ref{r-regular-aleph-1}, $\cu M_\cu D$ is $\aleph_1$-saturated.  Suppose that $\cu I=\langle a_h\rangle_{h\in\BN}$ is an
indiscernible $\omega$-sequence in $\cu M_\cu D$.

Let $r=d^{\cu M_{\cu D}}(a_0,a_1)/2$.  Then $r>0$ and $\cu I$ is $2r$-separated.
By Lemma \ref{l-maximal-indiscernible}, to prove that $\cu M_{\cu D}$ is $2^{|I|}$-saturated
it is enough to show that $\cu I$  can be extended to
an indiscernible sequence of cardinality $2^{|I|}$ in $\cu M_{\cu D}$.
Let $\cu M'=(\cu M, S,\in)$ be the two-sorted real-valued structure where one sort is $M$, the other sort is the set $S$ of
$\omega$-sequences $\bo b$ of elements of $M$ such that $d^{\cu M}(\bo b_i,\bo b_j)\ge r$ whenever $i<j$, $\bo z$ is a variable of sort $S$,
and $\in$ is the predicate on $M\times S$ such that the formula $a\in \bo z$ has value $0$ if $a$ occurs in $\bo z$, and has value $1$ otherwise.  (Again, $\cu M'$ is not a metric structure.)
By Remark \ref{r-regular-aleph-1}, $\cu M'_{\cu D}$ is $\aleph_1$-saturated.  For each $t\in I$, the sequence $\cu I[t]=\langle a_h[t]\rangle_{h\in\BN}$
is an element of sort $S$ in $\cu M'$.  Let $\cu I_{\cu D}$ be the corresponding element of sort $S$ in $\cu M'_{\cu D}$.

Let
$ \varphi_{\Delta,n,r}(\bo z)$ be the $V'$-formula saying that the sequence
$\bo z$ is $(\Delta,1/n)$-indiscernible and $r$-separated.  Let
$$\Gamma(\bo z)=\{a_h\in \bo z\colon h\in\BN\}\cup\{\varphi_{\Delta,n,r}(\bo z)\colon \Delta \mbox{ finite and strict}, n\in\BN\cap(2/r,\infty)\},$$
which is a countable set of $V'$-formulas.
\medskip

\textbf{Claim \ref{t-minimal=fsp}.1.}
Suppose $\bo b$ satisfies $\Gamma(\bo z)$ in $\cu M'_{\cu D}$ and
$$\cu H=\{a\in\cu M_{\cu D}\colon \cu M'_{\cu D}\models a\in\bo b\}.$$
Then $\cu H\supseteq\cu I$, $\cu H$ is indiscernible and $r$-separated in $\cu M_{\cu D}$, and $|\cu H|=2^{|I|}$.

\begin{proof} [Proof of Claim \ref{t-minimal=fsp}.1]
Since $\bo z$ has sort $S$, $\bo b[t]$ has sort $S$ in $\cu M'$ for each $t\in I$, and $\bo b$ is the element of the
ultraproduct corresponding to $\langle\bo b[t]\rangle_{t\in I}$.  The set of formulas $\Gamma(\bo z)$ guarantees that
$a_h\in\cu H$ for each $h\in\BN$ so $\cu H\supseteq\cu I$.  Also,
$\cu H$ is $(\Delta,1/n)$-indiscernible and $r$-separated in $\cu M_{\cu D}$ for each $n\in\BN$, and hence $\cu H$ is indiscernible and $r$-separated in $\cu M_{\cu D}$.
For each $t\in I$, the set
$$\cu H[t]=\{c\in\cu M\colon \cu M'\models c\in \bo b[t]\}$$
has cardinality $\aleph_0$.  Then by Fact \ref{f-cardinality-ultrapower}, we have
$$|\cu H|=|\BN_{\cu D}| = \aleph_0^{|I|} = 2^{|I|}.$$
\end{proof}

\textbf{Claim \ref{t-minimal=fsp}.2.}  $\Gamma(\bo z)$ is finitely satisfiable in $\cu M'_{\cu D}$.
\medskip

\begin{proof} [Proof of Claim \ref{t-minimal=fsp}.2]
For each $n\in\BN\cap(2/r,\infty)$ and finite $\Delta$, let
$$\Gamma_{n,\Delta}(\bo z)=\{a_h\in \bo z\colon h\le n\}\cup\{\varphi_{\Delta,n,r}(\bo z)\}.$$
Each $\Gamma_{n,\Delta}(\bo z)$ is finite.  Every finite subset of $\Gamma(\bo z)$ is implied by some $\Gamma_{n,\Delta}(\bo z)$,
because the formulas $\varphi_{\Delta,n,r}(\bo z)$ become stronger as $\Delta$ and $n$ increase.
Hence it suffices to prove that each $\Gamma_{n,\Delta}(\bo z)$ is satisfiable in $\cu M'_{\cu D}$.
Fix a finite $\Delta$ and $n\in\BN\cap(2/r,\infty)$.  By (iv), the conclusion of Lemma \ref{l-extend-indiscernible} holds.
Whenever $n\in\BN\cap(2/r,\infty]$ we have $r/2>1/n$.
So there exists $2\le k\in\BN$ such that whenever
$\Delta'\subseteq\Delta$,
every finite $r$-separated $(\Delta',1/(n+1))$-indiscernible set in $\cu M$ of cardinality $>k$ can be extended to a
$(\Delta',1/n)$-indiscernible set in $\cu M$ of cardinality $\aleph_0$.

There is a $V'$-formula $\theta(u_0,\ldots,u_k,\bo z)$ saying that $u_i\in \bo z$ for each $i\le k$, and $\vec u$ is $(\Delta,1/(n+1))$-indiscernible and $r$-separated.
Let $X$ be the set of $t\in I$ such that $\theta(a_0[t],\ldots,a_k[t],\cu I[t])=0$ in $\cu M'$.

Let $\psi(u_0,\ldots,u_k,\bo z)$
be the $V'$-formula saying that either $u_i\notin \bo z$ for some $i\le k$, or $d(u_i,u_j)\le r$ for some $i<j\le k$,
or $|\delta(\vec v)-\delta(\vec w)|\ge 1/(n+1)$ for some $\delta(\vec x)\in\Delta$
and some pair of $|\vec x|$-subtuples $\vec v,\vec w$ of $\vec u$.  Let $Y$ be the set of $t\in I$ such that $\psi(a_0[t],\ldots,a_k[t],\cu I[t])=0$ in $\cu M'$.

One can check that for every $(\vec u,\bo z)$ in $\cu M'$, either $\theta^{\cu M'}(\vec u,\bo z)=0$ or $\psi^{\cu M'}(\vec u,\bo z)=0$.
Therefore $X\cup Y=I$.  Since $\cu I$ is indiscernible and $2r$-separated in $\cu M_{\cu D}$, $\psi(\vec u,\cu I)$ has value $>0$ in $\cu M_{\cu D}$.
Therefore $Y\notin\cu D$, and hence $X\in\cu D$.
By the \Los \ Theorem, $\theta(a_0,\ldots,a_k,\cu I_{\cu D})=0$ in $\cu M'_{\cu D}$.
So $(a_0,\ldots,a_k)$ is $(\Delta,1/(n+1))$-indiscernible and $r$-separated in $\cu M_{\cu D}$.

For each $t\in X$, $(a_0[t],\ldots,a_k[t])$ has cardinality $>k$ and is $(\Delta,1/(n+1))$-indiscernible and $r$-separated in $\cu M$,
so $(a_0[t],\ldots,a_k[t])$ can be extended to a $(\Delta,1/n)$-indiscernible $\omega$-sequence $\bo c[t]$  in $\cu M$.
Therefore  $\Gamma(\bo c[t])$ holds in $\cu M'$ for each $t\in X$, and by the \Los \ Theorem, $\bo c_{\cu D}$ satisfies $\Gamma(\bo z)$ in $\cu M'_{\cu D}$.
This proves Claim  \ref{t-minimal=fsp}.2.
\end{proof}

Since $\Gamma(\bo z)$ is countable and finitely satisfiable in the $\aleph_1$-saturated structure $\cu M'_{\cu D}$, there is an element $\bo b$
that satisfies $\Gamma(\bo z)$ in $\cu M'_{\cu D}$.  By Claim \ref{t-minimal=fsp}.1,
$\cu H$ contains $\cu I$ and is indiscernible and $r$-separated in $\cu M_{\cu D}$, and $|\cu H|=2^{|I|}$.
Then by Lemma \ref{l-maximal-indiscernible} , $\cu M_{\cu D}$ is $2^{|I|}$-saturated.  As mentioned above, this  prove (i).
\end{proof}

\begin{cor}  \label{c-FO-minimal}  A first order theory has the NFCP if and only if it does not have the finite cover property as defined in [Ke67].
\end{cor}

\begin{proof}  Let $T$ be a complete first order theory.  By Theorem VI.5.8 (i) of [Sh78], $T$ is $\lek_\BF$-minimal if and only if $T$
does not have the FCP in the sense of [Ke67].   By Theorem \ref{t-minimal=fsp}, $T$ is $\lek$-minimal if and only if $T$ has the NFCP.
By Lemma \ref{l-min-max-G} (i), $\Sat(T)$ is the class of all regular ultrafilters iff $T$ is $\lek_\BF$-minimal, and also iff $T$ is $\lek$-minimal.
\end{proof}

\begin{thm}  \label{t-stable-notmin}
(i) A metric theory $T$ is stable but not $\lek$-minimal if and only if  $\Sat(T)=LCA$.

(ii)  There are exactly two $\lek$-equivalence classes of stable metric theories, the set of $\lek$-minimal theories and the
set of stable non-$\lek$-minimal theories.
\end{thm}

\begin{proof}

(i) By Theorem \ref{t-stable-vs-unstable}, $T$ is stable iff $\Sat(T)=LCA$ or $\Sat(T)\supsetneqq LCA$.  By
Theorem \ref{t-minimal=fsp}, $T$ is $\lek$-minimal iff $\Sat(T)\supsetneqq LCA$.  So (i) holds.

(ii)  If $T$ is $\lek$-minimal, then $T$ has the NFCP by Theorem \ref{t-minimal=fsp}, so $T$ is stable.
There are first order theories that are stable and not $\lek$-minimal, so the set of stable non-$\lek$-minimal theories is non-empty.
If $T$ and $U$ are both stable but not minimal, then by (i), $\Sat(T)=LCA=\Sat(U)$, so the set of stable non-$\lek$-minimal theories belongs to $\BM$.
\end{proof}

As corollaries, we identify the first two elements in the partial ordering $(\BM,\lek)$, and show that there is a $\cu D$ such that $\Sat_\BM(\cu D)=\{\min_\BM\}$.

\begin{cor}  \label{c-first-two}
For each $H\in\BM$, one of the following holds:
$$ H=  \miin_\BM,\qquad H=\snm_\BM,\qquad \snm_\BM\lk H.$$
\end{cor}

\begin{proof} We have $\min_\BM\in\BM$ by Lemma \ref{l-min-max-G}, and  $\snm_\BM\in\BM$ by Theorem \ref{t-stable-notmin}.
Therefore, $H\in\BM\setminus\{\min_\BM,\snm_\BM\}$ if and only if the theories in $H$ are unstable.
By Fact \ref{f-exists-D}, $LCF$ is a proper subclass of $LCA$, so by Theorem \ref{t-stable-vs-unstable}, if the theories in $H$ are unstable then $\snm_\BM\lk H$.
\end{proof}

\begin{cor}

$$\{\cu D\colon \Sat_\BM(\cu D)=\{\miin_\BM\}\}=\{\cu D\colon \Sat_\BF(\cu D)=\{\miin_\BF\}\}=\{\cu D\colon \cu D\notin LCA\}\ne\emptyset.$$
\end{cor}

\begin{proof}  It follows from Section VI of [Sh78]  that
$$\{\cu D\colon \Sat_\BF(\cu D)=\{\miin_\BF\}\}=\{\cu D\colon \cu D\notin LCA\}\ne\emptyset.$$
We prove
$$\{\cu D\colon \Sat_\BM(\cu D)=\{\miin_\BM\}\}=\{\cu D\colon \Sat_\BF(\cu D)=\{\miin_\BF\}\}.$$

Assume that $\Sat_\BF(\cu D)\ne\{\min_\BF\}$.  Then there exists $H\in \Sat_\BF(\cu D)$ such that $H\ne\min_\BF$.
By Corollaries \ref{c-tau} and \ref{c-tau-min-max}, $\tau(H)\ne\min_\BM$ and $\tau(H)\in\Sat_\BM(\cu D)$, so $\Sat_\BM(\cu D)\ne\{\min_\BM\}$.

Assume that $\Sat_\BF(\cu D)=\{\min_\BF\}$.  Then $\cu D\notin LCA$.   Suppose $T\in \Sat_\BM(\cu D)$.  Then $\Sat(T)\ne LCA$, so by Theorem \ref{t-stable-notmin} (i),
$T$ is either unstable or $\lek$-minimal.  But if $T$  is unstable, then $\cu D\in LCF$ by Theorem \ref{t-stable-vs-unstable} (ii),
contradicting $\cu D\notin LCA$, so $T$ is stable.  Therefore $T$ is $\lek$-minimal, so $\Sat_\BM(\cu D)=\{\min_\BM\}$.
\end{proof}

\section{$\lek$-maximal theories and SOP$_2$}

Shelah [Sh78] introduced a family of ``strict order properties'' of first order theories, including the property SOP$_2$.
Malliaris and Shelah [MS16a] proved that every first order theory with SOP$_2$ is maximal in
$(\BF,\lek)$, and conjectured that the converse also holds.    In this section we
introduce a metric analogue SOP$_2$, and show that metric theories that
that have SOP$_2$ are maximal in $(\BM,\lek)$.

We refer to  [Sh78], [MS15], and the expository article [Co12] for background about  SOP$_2$ and related notions in first order logic.

We now introduce some notation for working with binary trees.  We identify each $n\in\BN$ with the set $n=\{0,\ldots,n-1\}$.
Let $\{0,1\}^{<\BN}$ be the set of finite sequences of elements of $\{0,1\}$,  let $\{0,1\}^\BN$ be the set of infinite sequences of elements of $\{0,1\}$,
and let $\{0,1\}^n$ be the set of $n$-tuples of elements of $\{0,1\}$.
Thus $\{0,1\}^{<\BN}=\bigcup_n\{0,1\}^n$.
We regard the elements $\{0,1\}^{<\BN}$ as nodes of the full binary tree, and regard the elements of $\{0,1\}^\BN$ as branches of the full binary tree.
For $\sigma\in \{0,1\}^\BN$ and $n\in\BN$, let $\sigma|n=\langle \sigma(k)\rangle_{k<n}\in \{0,1\}^n$.
For $\mu,\nu\in \{0,1\}^{<\BN}$, note that $\mu\subseteq \nu$ iff  $\mu$ is an initial segment of $\nu$.
We call $\mu,\nu$ \emph{comparable} if $\mu\subseteq\nu$ or $\nu\subseteq\mu$, and \emph{incomparable} otherwise.

\begin{df} \label{d-SOP2}
A  metric theory $T$ has SOP$_2$ if there is a formula $\theta(\vec x,\vec z)$, called an SOP$_2$ formula for $T$,
such that in some model $\cu M\models T$ there are $|\vec z|$-tuples
$\langle \vec b_\mu\rangle_{\mu\in \{0,1\}^{<\BN}}$ such that for each $\sigma\in \{0,1\}^\BN$ and $n\in\BN$,
\begin{equation}  \label{eq-SOP1}
\cu M\models(\inf_{\vec x})\max\{\theta(\vec x,\vec b_{\sigma|k})\colon k<n\}=0,
\end{equation}
but for each  incomparable pair $\mu,\nu\in \{0,1\}^{<\BN}$ we have
\begin{equation}  \label{e-SOP2}
\cu M\models (\inf_{\vec x})\max(\theta(\vec x,\vec b_\mu),\theta(\vec x,\vec b_\nu))=1.
\end{equation}
\end{df}

\begin{rmk}    A first order theory  has SOP$_2$ in the above sense if and only if it has SOP$_2$ in the sense of [Sh78].
\end{rmk}

\begin{cor}  \label{c-SOP2-approx}
 $T$ has SOP$_2$ if and only if for some $0\le s<r\le 1$, Definition \ref{d-SOP2} holds  with (1) and (2) replaced by
\begin{equation}
\cu M\models(\inf_{\vec x})\max\{\theta(\vec x,\vec b_{\sigma|k})\colon k<n\}\le s,
\end{equation}
and
\begin{equation}
 \cu M\models (\inf_{\vec x})\max(\theta(\vec x,\vec b_\mu),\theta(\vec x,\vec b_\nu))\ge r.
 \end{equation}
We call $\theta$ an SOP$_2$-formula with bounds $(s,r)$ in $T$.
\end{cor}

\begin{proof}
Similar to the proof of Corollary \ref{c-local-unstable}.
\end{proof}

\begin{fact}  \label{f-FO-maximal} (Malliaris and Shelah [MS16a], Theorem 11.11). Every first order theory with SOP$_2$ is maximal in $(\BF,\lek)$.
\end{fact}

Let $\cu K_{bt}$ be the first order structure $(\{0,1\}^{<\BN},\subseteq)$ (a full binary tree).  The first order theory $\Th(\cu K_{bt})$ has SOP$_2$.
We will need the following characterization of good regular ultrafilters, which was shown and used in [MS16a] to prove Fact \ref{f-FO-maximal}.

\begin{fact}  \label{f-treetops}  (By Theorem 10.26 and Lemma 11.6 in [MS16a])
A regular ultrafilter $\cu D$ is good if and only if in the pre-ultrapower $\cu K_{bt}^\cu D$, every pairwise comparable set of cardinality $\le |I|$
has an upper bound.
\end{fact}

\begin{thm}  \label{t-SOP2}  Every metric theory $T$ with SOP$_2$ is maximal in $(\BM,\lek)$.
\end{thm}

\begin{proof}  Let $\cu D\in\Sat(T)$.
Then $\cu D$ is a regular ultrafilter over a set $I$, and for every model $\cu M$ of $T$, the pre-ultrapower  $\cu M^{\cu D}$ is $|I|^+$ saturated.
By Lemma \ref{l-min-max-G} (ii), it suffices to prove that $\cu D$ is good.
Suppose $A$ is a pairwise comparable subset of  $ \cu K_{bt}^\cu D$ of cardinality $\le|I|$. By Fact \ref{f-treetops} above, it is enough to
show that the set of formulas $\Gamma=\{a\subseteq u\colon a\in A\}$  is satisfiable in $\cu K_{bt}^\cu D$.
Since $\cu K_{bt}^\cu D\equiv\cu K_{bt}$, every finite pairwise comparable subset of $\cu K_{bt}^\cu D$ has a greatest element, so $\Gamma$ is finitely satisfiable in
$\cu K_{bt}^\cu D$.

There is an SOP$_2$-formula $\theta(\vec x,\vec z)$ for $T$, a model $\cu M$ of $T$, and $|\vec z|$-tuples
$\langle \vec b_\mu\rangle_{\mu\in \{0,1\}^{<\BN}}$ in $\cu M$ that satisfy Conditions (1) and (2) of Definition \ref{d-SOP2}.
For each $a\in A$, let $\vec b_a$ be the $|\vec z|$-tuple in $\cu M^\cu  D$ such that for each $t\in I$, $\vec b_a[t]=\vec b_{a[t]}$.
Consider a finite $A_0\subseteq A$.  By the \Los \ Theorem, there is a set $X_0\in\cu D$ such that for each $t\in X_0$ the set $\{a[t]\colon a\in A_0\}$
is pairwise comparable in $\cu K_{bt}$.  By Condition (1), for each $t\in X_0$ we have
$$\cu M\models (\inf_{\vec x})\max\{\theta(\vec x,\vec b_a[t])\colon a\in A_0\}=0,$$
so by the \Los \ Theorem,
$$\cu M^\cu D\models (\inf_{\vec x})\max\{\theta(\vec x,\vec b_a)\colon a\in A_0\}=0.$$
By Remark \ref{r-regular-aleph-1}, $\cu M^\cu D$ is $\aleph_1$-saturated, so there exists $\vec x$ in $\cu M^\cu D$ such that
$$\cu M^\cu D\models \max\{\theta(\vec x,\vec b_a)\colon a\in A_0\}=0.$$
Thus the set of formulas
$$\Gamma'=\{\theta(\vec x,\vec b_a)\colon a\in A\}$$
is finitely satisfiable in $\cu M^\cu D$.  Since $\cu M^\cu D$ is $|I|^+$-saturated, there is a tuple $\vec x$
that satisfies $\Gamma'$ in $\cu M^\cu D$.

Let $\cu X$ regularize $\cu D$, so $\cu X\subseteq\cu D$,  $|\cu X|=|I|$, and each $t\in I$ belongs to only finitely many $X\in\cu X$.
Pick an injective function $h\colon A\to\cu X$.
Consider an element $t\in I$. Let $A_t=\{a\in A\colon t\in h(a)\}$ be the finite set of $a\in A$ that ``matter'' for $t$.  Let
$$B_t=\{a\in A_t\colon \cu M\models \theta(\vec x[t],\vec b_a[t])<1\}.$$
By Condition (2), any two elements of $B_t$ are comparable in $\{0,1\}^{<\BN}$.   Let $u[t]=\bigcup B_t$, the maximum element of $B_t$ in $\cu K_{bt}$
(with the convention that $u[t]$ is the empty sequence if $B_t$ is empty).  Then $u\in\cu K_{bt}^\cu D$.

Fix an element $a\in A$.  Let
$$ Y= h(a)\cap\{t\in I\colon a\in B_t\}=h(a)\cap\{t\in I\colon \cu M\models \theta(\vec x[t],\vec b_a[t])<1\}.$$
By the \Los \ Theorem and the fact that $\cu M^\cu D\models \theta(\vec x,\vec b_a)=0$, we have $Y\in\cu D$.  For each $t\in Y$ we have
$\cu K_{bt}\models a[t]\subseteq u[t]$, so by the \Los \ Theorem again, $\cu K_{bt}^\cu D\models a\subseteq u$.  Therefore $u$ satisfies $\Gamma$ in $\cu K_{bt}^\cu D$,
and the proof is complete.
\end{proof}

\section{$\lek$-minimal unstable theories and the independence property}

The theory $T_{rg}$ of the random (or Rado) graph is the first order theory with a single binary relation $R$ whose axioms say that $R$ is symmetric and reflexive, and that for any two
disjoint finite sets of elements $X,Y$ there is an element $z$ such that $R(z,x)\wedge \neg R(z,y)$ for all $x\in X$ and $y\in Y$.  $T_{rg}$
is  $\aleph_0$-categorical, simple, unstable, and has quantifier elimination.

In Theorem \ref{t-rg-minimal} below, we  show that $T_{rg}$
is $\lek$-minimal among all unstable metric theories.
It will follow that the equivalence class of $T_{rg}$ is the third element of $(\BM,\lek)$ and is not maximal in $(\BM,\lek)$.
This is a continuous analogue of the result of Malliaris that  $T_{rg}$  is $\lek$-minimal among all unstable first order  theories (Fact \ref{f-rg-classical} below).

Shelah, in Theorem II.4.7 of [Sh78], proved that a first order theory is unstable if and only if it has the strict order property or the independence property.
Dzamonja and Shelah [DS04] proved that a first order theory with the strict order property has SOP$_2$, so a first order theory is unstable if and
only if it has either SOP$_2$ or the independence property.
In [BY09], Ben Yaacov introduced the continuous analogue of the independence property, which is equivalent to Definition \ref{d-indep} below.
We will show in Theorem \ref{t-SOP2-IP} below that a metric theory is unstable if and only if it has either the SOP$_2$
or the independence property.

\begin{df} \label{d-indep}
 $T$ has the \emph{independence property} if there is a formula $\varphi(\vec x,\vec y)$ and a model $\cu M$ of $T$ such that
for each $m\in\BN$ there are $|\vec x|$-tuples $\langle\vec a_Z\rangle_{Z\subseteq m}$ and $|\vec y|$-tuples $\langle\vec b_n\rangle_{n<m}$
such that for all $n<m$ and $Z\subseteq m$,
$$ n\in Z\Rightarrow \varphi^\cu M(\vec a_Z,\vec b_n)=0,\qquad n\notin Z\Rightarrow \varphi^\cu M(\vec a_Z,\vec b_n)=1.$$
We also say that $\varphi(\vec x,\vec y)$ has the independence property for $T$ (in $\cu M$).
\end{df}

\begin{rmk}  \label{r-indep}
If $\varphi(\vec x,\vec y)$ has the independence property for $T$, then for every $\aleph_1$-saturated model $\cu M$ of $T$,
$\varphi(\vec x,\vec y)$ has the independence property for $T$ in $\cu M$.
\end{rmk}

\begin{lemma}  \label{l-indep-approx}  $T$ has the independence property if and only if there exist $0\le s<r\le 1$ and a formula $\varphi(\vec x,\vec y)$
such that Definition \ref{d-indep} holds with the displayed formulas replaced by
$$ n\in Z\Rightarrow \varphi^\cu M(\vec a_Z,\vec b_n)\le s,\qquad n\notin Z\Rightarrow \varphi^\cu M(\vec a_Z,\vec b_n)\ge r.$$
\end{lemma}

\begin{proof}  Similar to the proof of Corollary \ref{c-local-unstable}.
\end{proof}

\begin{lemma} \label{l-indep-limit}  Let $\cu M$ be an $\aleph_1$-saturated model of $T$.  The following are equivalent:
\begin{itemize}
\item[(i)] $T$ has the independence property.
\item[(ii)] There are real numbers $0\le s<r\le 1$,  a formula $\varphi(\vec x,\vec y)$,
 a tuple $\vec a$ in $\cu M$, and an indiscernible sequence $\langle\vec b_n\rangle_{n\in\BN}$ in $\cu M$ such that
$\varphi^{\cu M}(\vec a,\vec b_n)\le s$ for all even $n$, and $\varphi^{\cu M}(\vec a,\vec b_n)\ge r$ for all odd $n$.
\item[(iii)] There is a formula $\varphi(\vec x,\vec y)$, a tuple $\vec a$ in $\cu M$, and an indiscernible sequence $\langle\vec b_n\rangle_{n\in\BN}$ in $\cu M$ such that
$\lim_{n\to\infty}\varphi^{\cu M}(\vec a,\vec b_n)$ does not exist.
\end{itemize}
\end{lemma}

\begin{proof}
It is trivial that (ii) implies (iii).  If (iii) holds, then it is easily seen that (ii) holds with $\langle \vec b_n\rangle_{n]\in\BN}$ replaced by
some infinite subsequence of $\langle \vec b_n\rangle_{n]\in\BN}$.  Assume (i).  Then there exist $0\le s<r\le 1$ and a formula $\varphi(\vec x,\vec y)$ satisfying the conditions
of  Lemma \ref{l-indep-approx}.   By $\aleph_1$-saturation and Ramsey's theorem, (ii) holds.

Assume (ii), and let $r,s$, $\varphi(\vec x,\vec y)$, $\vec a$, and  $\langle\vec b_n\rangle_{n\in\BN}$  be as in (ii).  Then
$$ n\in Z\Rightarrow \varphi^\cu M(\vec a,\vec b_{f(n)})\le s,\qquad n\notin Z\Rightarrow \varphi^\cu M(\vec a,\vec b_{f(n)})\ge r.$$
Since $\langle\vec b_n\rangle_{n\in\BN}$ is indiscernible in $\cu M$, it follows from Lemma \ref{l-indep-approx} that $T$ has the independence property, so (i) holds.
\end{proof}

In [BY09], the independence property was defined for metric theories  in a different and more complicated way.
Lemma 5.4 in [BY09] is the same as Lemma \ref{l-indep-limit} above except that $T$ is a metric theory, and (i) refers to the independence property as defined in [BY09].
 Thus a metric theory $T$ has the independence property in the sense of Definition \ref{d-indep} if and only if it has the independence
property in the sense of [BY09].
Here is a continuous analogue of Shelah's result that a first order theory is unstable iff it has the strict order property or the independence property.
\footnote{In [Kh19], Khanaki proved another continuous analogue of Shelah's result.  He introduced a property of metric theories called $wSOP$,
and showed that every unstable metric theory has either wSOP or the independence property.}

\begin{thm}  \label{t-SOP2-IP}  A  metric theory is unstable if and only if it has either  SOP$_2$ or the independence property.
\end{thm}

\begin{proof}  It is easily seen that if $T$ has the independence property, then $T$ is unstable.  If $T$ has SOP$_2$, then
$T$ is $\lek$-maximal by Theorem \ref{t-SOP2}, so $T$ is unstable by Theorem \ref{t-stable-vs-unstable}.

Suppose $T$ is unstable and does not have the independence property.  By Corollary \ref{c-local-unstable}, there is  a
formula $\varphi(\vec x,\vec y)$ that is unstable for $(0,1)$ in $T$.
Thus in some $\aleph_1$-saturated model
$\cu M\models T$ there is an infinite sequence $\langle \vec a_n,\vec b_n\rangle_{n\in\BN}$ of $|(\vec x,\vec y)|$-tuples such that whenever $h<k$,
$\varphi^{\cu M}(\vec a_h,\vec b_k)=0$ and $\varphi^{\cu M}(\vec a_k,\vec b_h)=1$.  We may also take $\langle \vec a_n,\vec b_n\rangle_{n\in\BN}$
so that $\varphi(\vec a_h,\vec b_h)=0$ for all $h$.

Let $\vec u=(\vec x,\vec y)$,  let $\vec v=(\vec x_*, \vec y_*)$ be a tuple of new variables with the same length as $\vec u$, and let $\vec c_h=(\vec a_h,\vec b_h)$.
Then $\theta(\vec u,\vec c_h)$ is the formula $\varphi(\vec x,\vec b_h)$, and $\theta(\vec c_h,\vec c_k)$ is the formula $\varphi(\vec a_h,\vec b_k)$.
 $\theta(\vec u,\vec v)$ is unstable for $(0,1)$ in $T$ with $|\vec u|=|\vec v|$,
and whenever $h<k$ we have $\theta^\cu M(\vec c_h,\vec c_k)=0$ and $\theta^{\cu M}(\vec c_k,\vec c_h)=1$.
Using Ramsey's theorem and $\aleph_1$-saturation, there is an indiscernible family $\cu I=\langle \vec c_m\rangle_{m\in\BQ}$ in $\cu M$
indexed by the set $\BQ$ of dyadic rationals in $(0,1)$
such that whenever $h<k$ we have $\theta^\cu M(\vec c_h,\vec c_k)=0$ and $\theta^{\cu M}(\vec c_k,\vec c_h)=1$, and $\theta^{\cu M}(\vec c_h,\vec c_h)=0$.

Let $\vec u\sqsubset\vec v$ be the formula
$$\vec u\sqsubset\vec v:\quad \max(\theta(\vec u,\vec v),1-\theta(\vec v,\vec u)).$$
Thus $\cu M\models\vec u\sqsubset\vec v = 0$ if and only if $ \theta^{\cu M}(\vec u,\vec v)=0$ and $ \theta^{\cu M}(\vec v,\vec u)=1$.
Let $\beta(\vec u,\vec t,\vec v)$ be the formula saying that $\vec t$ is ``between'' $\vec u$ and $\vec v$:
$$\beta(\vec u,\vec t,\vec v):\quad\max(\vec u\sqsubset\vec v,\vec u\sqsubset\vec t,\vec t\sqsubset\vec v).$$

For each $n\in\BN$, let $\psi_n(\vec t,\langle \vec u_m\rangle_{m<2^{n+1}})$ be the formula saying that
$\vec t$ is between $\vec u_m$ and $\vec u_{m+1}$ for each even $m<2^{n+1}$.
 Formally,
$$ \psi_n(\vec t,\langle \vec u_m\rangle_{m<2^{n+1}}): \quad \max_{h< 2^n}\beta(\vec u_{2h},\vec t,\vec u_{2h+1}).$$
By indiscernibility, for each $n\in\BN$ and increasing sequence $\langle h(0),\ldots,h(2^n-1)\rangle$ in $\BQ$, the sentence
$(\inf_{\vec t})\psi_n(\vec t,\langle \vec c_{h(m)}\rangle_{m<2^{n+1}})$ has the same value $r_n$ in $\cu M$.
It is clear that $r_0=0$, because $\beta^{\cu M}(\vec c_h,\vec c_k,\vec c_\ell)=0$ when $h<k<\ell$.

If there were infinitely many $n$ such that $r_n=0$, there would be an indiscernible subsequence $\langle \vec a_n\rangle_{n\in\BN}$ of $\cu I$
and a tuple $\vec t$ in $\cu M$ such that $\theta^{\cu M}(\vec t,\vec a_n)= 1$ for each even $n$ and $\theta^{\cu M}(\vec t,\vec a_n)= 0$ for each odd $n$.
But by Lemma \ref{l-indep-limit}, this would contradict the hypothesis that $T$ does not have the independence property.  Therefore there is a greatest $N\in\BN$
such that $r_N=0$, and thus $r_{N+1}>0$.

We will show that  $\psi_N$ is an SOP$_2$ formula with bounds $(0,r_{N+1})$ for $T$. Then by Corollary \ref{c-SOP2-approx}, it will follow that $T$ has SOP$_2$.

For each $h<k$ in $\BQ$,  let $\vec d_{h,k}$ be the sequence  with $2^{N+1}$ elements evenly spaced between $\vec c_h$ and $\vec c_k$,
so $\vec d_{h,k}=\langle \vec c_{h+m\delta}\rangle_{m<2^{N+1}}$
where $\delta=(k-h)/2^{N+1}$.  It follows that whenever $h_1<k_1<h_2<k_2$ in $\BQ$, we have
\begin{equation}  \label{e-SOP2-1}
\cu M\models (\inf_{\vec t})\psi_N(\vec t,\vec d_{h_1,k_1})=0,
\end{equation}
but
\begin{equation} \label{e-SOP2-2}
\cu M\models (\inf_{\vec t})\max[\psi_N(\vec t,\vec d_{h_1,k_1}),\psi_n(\vec t,\vec d_{h_2,k_2})]\ge r_{N+1}>0.
\end{equation}
For each node $\mu$ in the binary tree $\{0,1\}^{<\BN}$, pick dyadic rationals $h_\mu,k_\mu\in\BQ$ such that
$$h_\mu<h_{\mu0}<k_{\mu0}<h_{\mu1}<k_{\mu1}<k_\mu.$$
This gives us a nested binary tree of intervals in $\BQ$.  By (\ref{e-SOP2-1}) and (\ref{e-SOP2-2}), $\Psi_N$ is an SOP$_2$
formula with bounds $(0,r_N)$ in $T$ with the parameters $\vec b_\mu=\vec d_{h_\mu,k_\mu}$.
\end{proof}
\medskip

\begin{fact}  \label{f-rg-classical}  (Lemma 5.3 in Malliaris [Ma12]) $T_{rg}$  is $\lek$-minimal among all unstable first order theories,
that is, $T_{rg}\lek S$ for every unstable first order theory $S$.
\end{fact}

We now improve Fact \ref{f-rg-classical} by showing that $T_{rg}$ is $\lek$-minimal among all unstable metric theories.

\begin{thm}  \label{t-rg-minimal}  $T_{rg}\lek T$ for every unstable metric theory $T$.
\end{thm}

\begin{proof}  Suppose $T$ is unstable.    If $T$ has SOP$_2$,
then $T$ is $\lek$-maximal by Theorem \ref{t-SOP2}, and hence $T_{rg}\lek T$.

Now suppose $T$ does not have  SOP$_2$.  By Theorem \ref{t-SOP2-IP}, $T$ has the independence property.  Then there is a continuous formula
$\varphi(\vec x,\vec y)$ that has the independence property for $T$.  Let $\cu D\in\Sat(T)$, and $\cu M$ be an $\aleph_1$-saturated model of $T$.
Then $\cu M^\cu D$ is $|I|^+$-saturated.

Let $\cu K_{rg}$ be a first order structure which is a model of $T_{rg}$. We must show that $\cu K_{rg}^\cu D$ is $|I|^+$-saturated.  Let $A\subseteq\cu K_{rg}^\cu D$
have cardinality $|A|\le |I|$, and let  $p(z)$ be a set of formulas with parameters in $A$ that is maximal consistent in $\cu K_{rg}^\cu D$.
Since $T_{rg}$ has elimination of quantifiers, and its vocabulary has only equality and the binary predicate symbol $R$ which is symmetric in $T_{rg}$,
two elements $z_1,z_2$ of $\cu K_{rg}^{\cu D}$ have the same type over $A$ if and only if
$$\{a\in A\colon \cu K_{rg}^{\cu D}\models R(z_1,a)\}=\{a\in A\colon \cu K_{rg}^{\cu D}\models R(z_2,a)\}.$$
Therefore there are sets $B,C$ such that $A=B\cup C$ and $p(z)$ is equivalent in $\cu K_{rg}^{\cu D}$ to the set of formulas
$$\Gamma=\Gamma(x,A)=\{ R(z,b)\colon b\in B\}\cup\{\neg R(z,c)\colon c\in C\}.$$
Then $\Gamma$ is finitely satisfiable in $\cu K_{rg}^\cu D$, and $\cu K_{rg}^\cu D\models b\ne c$ for all $b\in B$ and $c\in C$.
To complete the proof it suffices to show that $\Gamma$ is satisfiable in $\cu K_{rg}^\cu D$.

As in the proof of Theorem \ref{t-SOP2}, let $\cu X$ regularize $\cu D$, and pick an injective function $h\colon A\to\cu X$.
For each $t\in I$, the sets
$$A_t=\{a\in A\colon t\in h(a)\}\subseteq\cu K_{rg}^\cu D, \qquad A[t]=\{a[t]\colon a\in A, t\in h(a)\}\subseteq\cu K_{rg} $$
are finite.
Since $\varphi(\vec x,\vec y)$ has the independence property for $T$, for each $a\in A$ we may choose a $|\vec y|$-tuple $ \hat a$ in $\cu M^\cu D$
such that for each $t\in I$ and set $Z\subseteq A[t]$ there is a tuple $\vec x_Z$ in $\cu M$ for which
$$a[t]\in Z\Rightarrow\varphi^{\cu M}(\vec x_Z,\hat a[t])=0, \qquad a[t]\in A[t]\setminus Z\Rightarrow \varphi^{\cu M}(\vec x_Z,\hat a[t])=1.$$
Let $\Gamma'$ be the set of continuous formulas
$$\Gamma'=\{\varphi(\vec x,\hat b)=0,\varphi(\vec x,\hat c)=1\colon b\in B, c\in C\}.$$
Let $B_0, C_0$ be finite subsets of $B, C$.
By the \Los \ Theorem applied to  $\cu K_{rg}^{\cu D}$, the set
$$X=\{t\in I\colon(\forall b\in B_0)(\forall c\in C_0) b[t]\ne c[t]\}$$
belongs to $\cu D$.  For each $t\in X$ there is a tuple $\vec x$ in $\cu M$ such that
$$(\forall b\in B_0)\varphi^\cu M(\vec x,\hat b[t])=0, \quad (\forall c\in C_0)\varphi^\cu M(\vec x,\hat c[t])=1.$$
Hence by the \Los \ Theorem, $\Gamma'$ is finitely satisfiable in $\cu M^\cu D$.
Since $\cu M^\cu D$ is $|I|^+$-saturated, there is a tuple
$\vec x$ that satisfies $\Gamma'$ in $\cu M^\cu D$.

We now show that $\Gamma$ is satisfiable in $\cu K_{rg}^\cu D$ by finding an element $z\in\cu K_{rg}^\cu D$ that ``matches'' $\vec x$.  Fix a real number $0<r<1/2$.
Let $t\in I$.
Note that if $b,c\in A$, $\varphi^\cu M(\vec x[t],\hat b[t])\le r$, and $\varphi^\cu M(\vec x[t],\hat c[t])\ge 1-r$, then $b[t]\ne c[t]$.
So there is an element $z[t]\in\cu K_{rg}$ such that for each $a\in A_t$:
\begin{itemize}
\item $\varphi^\cu M(\vec x[t],\hat a[t])\le r \Rightarrow \cu K_{rg}\models R(z[t],a[t])$.
\item $\varphi^\cu M(\vec x[t],\hat a[t])\ge 1-r \Rightarrow \cu K_{rg}\models \neg R(z[t],a[t])$.
\end{itemize}
Since $\vec x$ satisfies $\Gamma'$ in $\cu M^\cu D$, for each $b\in B$ and $c\in C$, the set
$$Y=\{ t\in I\colon t\in h(b)\cap h(c), \varphi^\cu M(\vec x[t],\hat b[t])\le r, \varphi^\cu M(\vec x[t],\hat c[t])\ge 1-r\}$$
belongs to $\cu D$.  The element $z[t]$ was chosen so that the set
$$\{ t\in I\colon \cu K_{rg}\models R(z[t],b[t])\wedge \neg R(z[t],c[t])\}$$
contains $Y$, and thus also belongs to $\cu D$.  Therefore, by the \Los \ Theorem, $\cu K_{rg}^\cu D\models R(z,b)\wedge \neg R(z,c)$.
This shows that $z$ satisfies $\Gamma$ in $\cu K_{rg}^\cu D$, so $\cu D\in\Sat(T_{rg})$ and  $T_{rg}\lek T$.
\end{proof}

\section{$\lek$-minimal  TP$_2$ theories}

In Theorem  \ref{t-min-nonsimple} and Corollary \ref{c-dividingline} below, we will prove that there are metric theories $T$ that are
$\lek$-minimal among metric TP$_2$ theories and  strictly above $T_{rg}$.  These results are the continuous analogues of the first order Facts \ref{f-minimal-nonsimple}
 and \ref{f-dividingline}.
Also, in Theorem \ref{t-min-rgR}, we will give a natural example of such a theory $T$.
That example is  $T^R_{rg}$, the randomization of $T_{rg}$.
Our proofs  will use the notion of a distribution of a set of continuous formulas from Section \ref{s-general} above, as well as the earlier notion
of a distribution of a set of first order formulas introduced by Malliaris in [Ma12].

 Shelah [Sh80] defined simple theories in first order logic and proved that a first order theory is simple if and only if it has neither SOP$_2$
nor TP$_2$.
The theory $T^*_{feq}$ is the model completion of the first order theory of infinitely many parameterized equivalence relations
(see [DS04] for the precise definition).
$T^*_{feq}$ is TP$_2$ but not SOP$_2$, and thus is not simple.

\begin{fact} \label{f-minimal-nonsimple} (Corollary 6.10 in Malliaris  [Ma12])
$T^*_{feq}$  is $\lek$-minimal among  first order theories  with TP$_2$.
\end{fact}

\begin{fact}  \label{f-dividingline} (Malliaris and Shelah [MS13])  $T_{rg}\lk S$ for any  first order theory $S$ that has TP$_2$.
\end{fact}

The following  definition is equivalent to the first order notion in the case that $T$ is a first order theory.

\begin{df}  (Ben Yaacov [BY13])  \label{d-TP2}
A metric theory $T$ has the \emph{tree property of the second kind} (briefly, $T$ is TP$_2$),  if in some $\aleph_1$-saturated model $\cu M$ of $T$
there is a continuous formula $\varphi(\vec x,\vec y)$ and an array $\langle \vec b_{n,m}\rangle_{n,m\in\BN}$ of $|\vec y|$-tuples such that:
\begin{itemize}
\item[(a)] The sequences $\cu I_n=\langle\vec b_{n,m}\rangle_{m\in\BN}$ are mutually indiscernible.
\item[(b)] The sequence of sequences $\langle \cu I_n\rangle_{n\in\BN}$ is indiscernible.
\item[(c)] For each $n\in\BN$, $\{\varphi(\vec x,\vec b_{n,m})\colon m\in\BN\}$ is not satisfiable in $\cu M$.
\item[(d)] For each function $f\colon \BN\to\BN$, $\{\varphi(\vec x,\vec b_{n,f(n)})\colon n\in\BN\}$ is satisfiable in $\cu M$.
\end{itemize}
\end{df}

 Ben Yaacov [BY03] defined simple theories in the setting of compact abstract theories.
[BY13] points out that the definition can be
translated into the setting of continuous logic, and that no simple metric theory is TP$_2$.
It was shown in [MS20] that there are continuum many  $\lek$-equivalence classes of simple first order  theories in $(\BF,\lek)$.
Every simple first order theory is simple as a metric theory, so there are continuum many $\lek$-equivalence
classes of simple metric theories in $(\BM,\lek)$.  In this section we will avoid the translation from compact abstract theories to metric theories,
and  work directly with the continuous properties SOP$_2$ and TP$_2$.

\begin{lemma}  \label{l-TP2-approx}  Let $0\le \varepsilon < 1$.  $T$ is TP$_2$ if and only if  Definition \ref{d-TP2}
holds with $\varphi$ replaced by $\varphi\dotminus \varepsilon$ in conditions (c) and (d).
\end{lemma}

\begin{proof}  Say that $\varphi(\vec x,\vec y)$ is an $\varepsilon$-TP$_2$ formula if Definition \ref{d-TP2}
holds with  $\varphi$ replaced by $\varphi\dotminus\varepsilon$ in conditions (c) and (d).  If $\varphi$ is an $\varepsilon$-TP$_2$ formula, then
$\varphi\dotminus\varepsilon$ is a $0$-TP$_2$ formula.  Suppose  $\varphi$ is a $0$-TP$_2$ formula.
By indiscernibility and Remark \ref{r-dotminus}, there exists $0<\delta\le 1$ such that whenever $n,m,m'\in\BN$ and $m<m'$,
$$\cu M\models(\inf_{\vec x})\max[\varphi(\vec x,\vec b_{n,m}),\varphi(\vec x,\vec b_{n,m'})]=\delta.$$
If $\varepsilon < \delta$, then $\varphi$ is an $\varepsilon$-TP$_2$ formula.  On the other hand, if $\varepsilon\ge\delta$,
then $g(\varphi)$ is an $\varepsilon$-TP$_2$ formula, where $g\colon[0,1]\to[0,1]$ is the unary connective whose graph is the line
from $(0,0)$ to $(\delta,(\varepsilon+1)/2)$ followed by the horizontal line from $(\delta,(\varepsilon+1)/2)$ to $(1,(\varepsilon+1)/2))$.
\end{proof}

\begin{thm} \label{t-min-nonsimple}
$T^*_{feq}$ is $\lek$-minimal among TP$_2$ metric theories.
\end{thm}

Theorem \ref{t-min-nonsimple} is the continuous analogue of Fact \ref{f-minimal-nonsimple}.  To prove it we will use
Fact \ref{f-solves} below, which was proved in [Ma12] along the way to proving Fact \ref{f-minimal-nonsimple}.
Fact \ref{f-solves} will allow us to get around the difficulty that continuous logic does not have negation.

\begin{df} \label{d-solves}
Say that $\cu D$ \emph{solves} $(\omega,\omega)$ if for every first order structure $\cu K$,  first order formula $\varphi(\vec x,\vec y)$, and
array $C=\langle c_{n,m}\rangle_{n,m\in\BN}$ of elements of $\cu K$ such that:
\begin{itemize}
\item For all $n,m,m'$ with $m\ne m'$, $\cu K\models \neg[(\exists x)\varphi(\vec x,\vec c_{n,m})\wedge\varphi(\vec x,c_{n,m'})]$,
\item For each  $f\colon \BN\to\BN$, the type $\{\varphi(\vec x,\vec c_{n,f(n)})\colon n\in\BN\}$ is realized in $\cu K$,
\end{itemize}
and for every set $A\subseteq C^{\cu D}$ of cardinality $|A|\le\lambda$, if $\{\varphi(\vec x,\vec a)\colon \vec a\in A\}$
is finitely satisfiable in $\cu K^{\cu D}$ then it is satisfiable in $\cu K^{\cu D}$.
\end{df}

\begin{fact}  \label{f-solves}  (By Lemmas 6.7 and 6.8 in [Ma12]).
If $\cu D$ solves $(\omega,\omega)$ then $\cu D$ saturates $T^*_{feq}$.
\end{fact}

\begin{lemma}   \label{l-TP2-solves}  If $T$ is a metric TP$_2$ theory and $\cu D$ saturates $T$, then $\cu D$ solves $(\omega,\omega)$.
\end{lemma}

\begin{proof}  Let $\cu M$ be an $\aleph_1$-saturated model of $T$.  There is a continuous formula $\theta(\vec u,\vec v)$ and an array of
$|\vec u|$-tuples $\langle \vec b_{n,m}\rangle_{n,m\in\BN}$ in $\cu M$ such that:
\begin{itemize}
\item[(a)] For all $n,m,m'$ with $m\ne m'$, $\cu M\models (\inf_{\vec u})\max[\theta(\vec u,\vec b_{n,m}),\theta(\vec u,b_{n,m'})]=1$,
\item[(b)] For each  $f\colon \BN\to\BN$, the type $\{\theta(\vec u,\vec b_{n,f(n)})=0\colon n\in\BN\}$ is realized in $\cu M$.
\end{itemize}
Let $\cu K, \varphi(\vec x,\vec y)$, and $C=\langle c_{n,m}\rangle_{n,m\in\BN}$
be as in Definition \ref{d-solves}.  Suppose
$$A=\{\vec a_\beta\colon\beta<\lambda\}\subseteq C^{\cu D},$$
and
$$\Gamma(\vec x,A)= \{\varphi(\vec x,\vec a_\beta)\colon \beta<\lambda\}$$
is finitely satisfiable in $\cu K^{\cu D}$.  To show that $\cu D$ solves $(\omega,\omega)$, we must show that $\Gamma$ is satisfiable in $\cu K^{\cu D}$.
For each $\beta<\lambda$ and $t\in I$, we have
$\vec a_\beta[t]=\vec c_{n_\beta[t],m_\beta[t]}\in C$ for some $n_\beta[t],m_\beta[t]\in\BN$.
For each $\Phi\in\cu P_{\aleph_0}(\Gamma)$ and $t\in I$, let
$$\Phi[t]=\{\varphi(\vec x,\vec a_\beta[t])\colon\varphi(\vec x,\vec a_\beta)\in\Phi\}.$$
For each $\alpha<\beta<\lambda$, let
$$X_{\alpha,\beta}=\{t\in I\colon n_\alpha[t]=n_\beta[t]\Rightarrow m_\alpha[t]=m_\beta[t]\}.$$
By \ref{d-solves}, for each $\Phi\in\cu P_{\aleph_0}(\Gamma)$ and $t\in I$,
$\Phi[t]$ is satisfiable in $\cu K$ if and only if for all $\varphi(\vec x,\vec a_\alpha),\varphi(\vec x,\vec a_\beta)\in\Phi$ we have $t\in X_{\alpha,\beta}$,
Since $\Gamma$ is finitely satisfiable in $\cu K^{\cu D}$. By the \Los \ Theorem, for each $\alpha,\beta<\lambda$ we have $X_{\alpha,\beta}\in\cu D$.

Let $\vec b_\beta[t]=\vec b_{n_\beta[t],m_\beta[t]}$,
$\vec b_\beta=\langle \vec b_{n_\beta[t],m_\beta[t]}\rangle_{t\in I}\in\cu M^{\cu D}$, and $B=\{\vec B_\beta\colon\beta<\lambda\}$.
Let
$$\Sigma=\Sigma(\vec u,B)=\{\theta(\vec u,\vec b_\beta)\colon \beta<\lambda\}.$$
For $\Theta\in\cu P_{\aleph_0}(\Sigma)$, define $\Theta[t]$ similarly to $\Phi[t]$ above.
By (a) and (b), $\Theta[t]$ is satisfiable in $\cu M$ if and only if
if for all $\theta(\vec u,\vec b_\alpha),\theta(\vec u,\vec b_\beta)\in\Theta$ we have $t\in X_{\alpha,\beta}$.
Then by the \Los \ Theorem, $\Sigma(\vec u ,B)$
is finitely satisfiable in $\cu M^{\cu D}$.  Since $\cu D$ saturates $T$, $\Sigma(\vec u,B)$ is satisfied in $\cu M^{\cu D}$.
By Lemma \ref{l-dist-saturates}, $\Sigma$ has a multiplicative distribution $\delta_1\colon \cu P_{\aleph_0}(\Sigma^{ap})\to\cu D$ in $\cu M^{\cu D}$.
Let $0<s<1/2$.  Then for each formula $\theta\in \Sigma(\vec u,B)$, the formula $\theta\dotminus s$ belongs to $\Sigma^{ap}$.  For each $\Theta\in\cu P_{\aleph_0}(\Sigma)$, let
$$\delta_2(\Theta)=\delta_1(\{\theta\dotminus s\colon \theta\in\Theta\}).$$
$\delta_2$ is not a continuous distribution because it is defined on finite subsets of $\Sigma$ rather than $\Sigma^{ap}$.  However,
\begin{itemize}
\item $\delta_2\colon\cu P_{\aleph_0}(\Sigma)\to \cu X$ for some $\cu X$ that regularizes $\cu D$.
\item  For each $\Theta\in\cu P_{\aleph_0}(\Sigma)$ and $t\in \delta_2(\Theta)$,
$$\cu M\models(\inf_{\vec u})\max_{\theta\in\Theta}\theta(\vec u,B[t])\le s.$$
\item $\delta_2$ is multiplicative.
\end{itemize}
By (a) and (b), for each $\alpha<\beta<\lambda$ and $t\in\delta_2(\{\theta(\vec u,\vec b_\alpha),\theta(\vec u,\vec b_\alpha)\})$,
we still have $t\in X_{\alpha,\beta}$.

Finally, let $\delta_3\colon\cu P_{\aleph_0}(\Gamma)\to\cu D$ be the unique multiplicative mapping such that
$\delta_3(\{\varphi(\vec x,\vec a_\beta)\})=\delta_2(\{\theta(\vec u,\vec b_\beta)\})$ for each $\beta<\lambda$.
Then for each $\alpha<\beta<\lambda$ and $t\in\delta_3(\{\varphi(\vec x,\vec a_\alpha),\varphi(\vec x,\vec a_\alpha))$,
we have $t\in X_{\alpha,\beta}$. Therefore, for each $\Phi\in\cu P_{\aleph_0}(\Gamma)$, if $t\in\delta_3(\Phi)$ then
$\Phi[t]$ is satisfiable in $\cu K$.  This shows that $\delta_3$ is a first order distribution of $\Gamma$ in $\cu K^{\cu D}$,
so by Lemma \ref{l-dist-saturates}, $\Gamma$ is satisfiable in $\cu K^{\cu D}$.
\end{proof}

\begin{proof} [Proof of Theorem \ref{t-min-nonsimple}]
Let $T$ be a  metric TP$_2$ theory.
If $\cu D$ saturates $T$, then $\cu D$ solves $(\omega,\omega)$ by Lemma \ref{l-TP2-solves}, so $\cu D$
saturates $T^*_{feq}$ by Fact \ref{f-solves}.  Therefore  $T^*_{feq}\lek T$.
\end{proof}

As a corollary, we get the continuous analogue of Fact \ref{f-dividingline}.

\begin{cor}  \label{c-dividingline}
$T_{rg}\lk T$ for every TP$_2$ metric theory $T$.
\end{cor}

\begin{proof}  $T_{rg}\lk T^*_{geq}$ by Fact \ref{f-dividingline}, and $T^*_{geq}\lek T$ by Theorem \ref{t-min-nonsimple}, so $T_{rg}\lk T$.
\end{proof}

\begin{cor}  There is a regular ultrafilter $\cu D$ that saturates $T_{rg}$, but saturates no TP$_2$ metric theory.
\end{cor}

\begin{proof}  By Theorem 12.1 of [MS13], there is a $\cu D$ that saturates $T_{rg}$ but does not saturate $T^*_{feq}$.
Then by Theorem \ref{t-min-nonsimple}, $\cu D$ saturates no TP$_2$ metric theory.
\end{proof}

\begin{df}  $\cu D$ is OK if every monotone function $g\colon\cu P_{\aleph_0}(I)\to\cu D$ such that $g(u)=g(v)$ whenever $|u|=|v|$
has a multiplicative refinement $f\colon\cu P_{\aleph_0}(I)\to\cu D$.
\end{df}

Clearly, every good ultrafilter is OK.  OK ultrafilters were mentioned without a name in [Ke67].  In [Ku78], Kunen introduced the name OK ultrafilter
and studied them from a topological viewpoint.

\begin{fact}  \label{f-OK-TP2} (By Lemma 8.8 of [Ma12] and Claim 6.1 in [MS15])
If $\cu D$ saturates some TP$_2$ first order theory, then $\cu D$ is OK.
\end{fact}

\begin{cor}
If  $\cu D$ saturates some TP$_2$ metric theory, then $\cu D$ is OK.
\end{cor}

\begin{proof}  If $\cu D$ saturates $T$, then $\cu D$ saturates $T^*_{feq}$ by Theorem \ref{t-min-nonsimple},
so $\cu D$ is OK by Fact \ref{f-OK-TP2}.
\end{proof}

We now review the \emph{randomization} $S^R$ of a first order theory $S$.  $S^R$ is the complete metric theory  defined as follows (see [Ke99] and [BK09]).

Given a countable first order structure $\cu K$, ${\cu K}^{[0,1]}$ is the pre-metric structure with two sorts, the \emph{random element} sort $\BK$ whose universe is the
set of all Borel functions from $[0,1]$ into $\cu K$, and the \emph{event} sort $\BE$ whose universe is the set of all Borel subsets of $[0,1]$.
$\cu K^{[0,1]}$ has the Boolean operations in sort $\BE$, a unary predicate $\mu$ of sort $\BE$ for measure, and for each
first order formula $\varphi(\vec v)$, a function $\l \varphi(\vec v)\rr$ of sort $\BK^{|\vec v|}\to\BE$. $\l\varphi(\vec a)\rr$ is the
set $\{t\in[0,1]\colon \cu K\models\varphi(\vec a(t))\}$, and for each Borel set $X$, $\mu(X)$ is the Lebesgue measure of $X$.  If $\cu K_1\equiv\cu K_2$, then
$\cu K_1^{[0,1]}\equiv \cu K_2^{[0,1]}$, so there is a unique complete metric theory $S^R$, the \emph{randomization} of $S$,
such that for each countable model $\cu K$ of $S$, $\cu K^{[0,1]}$ is a model of $S^R$.  $S^R$ is a metric theory with the  distance predicate
$d(x,y)=\mu(\l x\ne y\rr)$.

\begin{fact} \label{f-randomization-indep}  Let $S$ be a first order theory.

(i)  $S^R$ is stable if and only if $S$ is stable.  (Ben Yaacov, Theorem 5.14 of [BK09]).

(ii)  $S^R$ has the independence property if and only if $S$ has the independence property.  (Theorem 4.10 of [BY13]).

(iii) If $S$ has the independence property, then $S^R$ is TP$_2$.  (Theorem 4.13 of [BY13]).
Thus if $S$ is simple but unstable then $S^R$ is  not simple.
\end{fact}

Since the randomization $S^R$ of a first order theory $S$ shares many properties with $S$, one might expect that $S^R$ is $\lek$-equivalent to $S$.
However, this is not the case when $S=T_{rg}$.

\begin{cor} \label{c-strict-rg} $T_{rg}\lk T_{rg}^R$.
\end{cor}

\begin{proof} $T_{rg}^R$ is TP$_2$
by Fact \ref{f-randomization-indep} (iii). So
$T_{rg}\lk T_{rg}^R$ by Corollary \ref{c-dividingline}.
\end{proof}

\begin{thm}  \label{t-min-rgR}
 $T^R_{rg}$ is $\lek$-minimal among TP$_2$ metric theories.
\end{thm}

\begin{proof}
Let $\cu N$ be a model of $T^R_{rg}$.  Let $T$ be a TP$_2$ metric theory,
$\cu M$ be an $\aleph_1$-saturated model of $T$, and $\cu D$ be a regular ultrafilter over a set $I$ of power $\lambda$ such that $\cu M^\cu D$ is $\lambda^+$-saturated.
We must show that $\cu N^{\cu D}$ is $\lambda^+$-saturated.

Let  $A$  be an infinite subset of  $\cu N^{\cu D}$ of cardinality $|A|\le\lambda$, and let $\Gamma=\Gamma(u,A)$ be a  set of continuous formulas with
one variable $u$ and parameters from $A$ that is complete in $T^R_{rg}$.  It suffices to show that $\Gamma(u,A)$ is satisfiable in $\cu N^{\cu D}$.

For each finite subset $F\subseteq A$, the restriction of $\Gamma$ to formulas with parameters from $F$ is a complete type $\Gamma_F$ in $u$ over $F$.
By Theorem 2.9 of [BK09], $T^R_{rg}$ has quantifier elimination.
Since $T_{rg}$ also has quantifier elimination and its vocabulary has only the single symmetric binary predicate $R$, $\Gamma_F$
is equivalent to the following  continuous (but not strict) formula $\gamma_{F}(u)$ with parameters in $F$:  Informally, for some probability measure $\nu_F$ over $F$, the
formula $\gamma_{F}(u)$ says that for each $Z\subseteq F$, $\nu_F(Z)$ is the $\mu$-measure of the set $Z=\{a\in F\colon R(u,a)\}$. Formally,
$$\gamma_F(u):= (\forall Z\subseteq F)\ \nu_{F}(Z)=\mu(\l\bigwedge_{a\in Z}R(u,a)\wedge\bigwedge_{a\notin Z} \neg R(u,a)\rr).$$
Therefore we may take $\Gamma(u,A)$ to be the set of formulas
$$\Gamma(u,A)=\{\gamma_{F}(u)\colon F\in\cu P_{\aleph_0}(A)\}.$$
Since $\Gamma(u,A)$ is complete in $T^R_{rg}$, it is finitely satisfiable in $\cu N^{\cu D}$.
By Lemma \ref{l-dist-saturates}, to show that
$\Gamma(u,A)$ is satisfiable in $\cu N^{\cu D}$, it suffices to show that $\Gamma$ has a multiplicative distribution in $\cu N^{\cu D}$.

Let $\kappa=|A|$.  Then $\kappa\le\lambda$.
Since $\cu M^{\cu D}$ is $\lambda^+$-saturated and $T$ is TP$_2$, in $\cu M^{\cu D}$ there is a formula $\varphi(\vec x,\vec y)$ and family
$\langle \vec b_{\alpha,\beta}\rangle_{\alpha,\beta\in\kappa\times\kappa}$ of $|\vec y|$-tuples in $\cu M^{\cu D}$
such that:
\begin{itemize}
\item For each $\alpha<\kappa$ and $\beta<\beta'<\kappa$
$$\cu M^{\cu D}\models \max(\varphi(\vec x,\vec b_{\alpha,\beta}), \varphi(\vec x,\vec b_{\alpha,\beta'}))=1.$$
\item For each function $f\colon \kappa\to \kappa$ and $G\in\cu P_{\aleph_0}(\kappa)$,
$$\cu M^{\cu D}\models (\inf_{\vec x})\max_{\alpha\in G}\varphi(\vec x,\vec b_{\alpha,f(\alpha)})=0.$$
\end{itemize}

Fix a function $f\colon \kappa\to \kappa$.  Let $\vec b_\alpha=\vec b_{\alpha,f(\alpha)}$ for each $\alpha<\kappa$, and $B=\{\vec b_\alpha\colon \alpha<\kappa\}$.
Let $\Sigma=\Sigma(\vec x,B)=\{\varphi(\vec x,\vec b_\alpha)\colon \alpha<\kappa\}$. Then $\Sigma$ is finitely satisfiable in $\cu M^{\cu D}$.
$\Gamma^{ap}$ and $\Sigma^{ap}$ both have cardinality $\kappa$, so there is a bijection $g$ from $\Gamma^{ap}$ onto $\Sigma^{ap}$.
  Since $\cu M^{\cu D}$ is $\lambda^+$-saturated, $\Sigma$ is satisfiable  in $\cu M^{\cu D}$.
By Lemma \ref{l-dist-exist}, $\Gamma(u,A)$ has an accurate distribution $\delta_1\colon\cu P_{\aleph_0}(\Gamma^{ap})\to\cu D$ in $\cu N^{\cu D}$.
By the proof of Lemma \ref{l-dist-exist}, $\Sigma$ has a unique accurate distribution $\delta_2$ in $\cu M^{\cu D}$ such that
for each $\Psi\in\cu P_{\aleph_0}(\Gamma^{ap})$, $\delta_2(g(\Psi))=\delta_1(\Psi)$.

By Lemma \ref{l-dist-saturates},  $\Sigma$ has a multiplicative distribution $\delta_3$ in $\cu M^{\cu D}$ that is a refinement of $\delta_2$.
Then
\begin{itemize}
\item $\delta_3\colon \cu P_{\aleph_0}(\Sigma^{ap})\to\cu X$ where $\cu X$ regularizes $\cu D$.
\item For each $\Psi\in\cu P_{\aleph_0}(\Sigma^{ap})$ and $t\in\delta_3(\Psi)$,
$$\cu M\models(\inf_{\vec x})\max\{\psi(\vec x,B[t])\colon \psi\in\Psi\}=0.$$
\item $\delta_3(\Psi\cup\Theta)=\delta_3(\Psi)\cap\delta_3(\Theta)$ for all $\Psi,\Theta\in\cu P_{\aleph_0}(\Sigma^{ap})$.
\end{itemize}
It follows that the mapping $\delta_4$ such that $\delta_4(\Psi)=\delta_3(g(\Psi))$ is a  multiplicative refinement of $\delta_1$,
and hence is a multiplicative distribution of $\Gamma$ in $\cu N^{\cu D}$.
\end{proof}

\begin{cor} \label{c-rgR} $T^R_{rg}\eqk T^*_{feq}$.
\end{cor}

\begin{proof}  By Theorems \ref{t-min-nonsimple} and \ref{t-min-rgR}.
\end{proof}

We conclude with an open question.
In  [MS16b], Theorem 11.4, Malliaris and Shelah proved that if there exists an uncountable compact cardinal then for all first order theories $S,U$,
if $S$ is simple and $U\lek S$ then $U$ is simple. Recall that a first order theory is simple if and only if it has neither SOP$_2$ nor TP$_2$.

\begin{question}
Suppose there exists an uncountable compact cardinal.  If $T,U$ are metric theories,  $U\lek T$, and $T$ has neither SOP$_2$ nor TP$_2$,
must $U$ have neither SOP$_2$ nor TP$_2$?
\end{question}

A potential path to an affirmative answer would be to generalize the proof in [MS16b] to continuous logic.

\end{document}